\documentclass[a4paper,12pt,reqno]{amsart}
\usepackage[a4paper,margin=0.9in]{geometry}
\usepackage[T1]{fontenc}
\usepackage{lmodern}
\usepackage{amsmath,amssymb,amsthm,mathtools}
\usepackage{esint}
\usepackage{microtype}

\AtBeginDocument{\begingroup\footnotesize\fontencoding{OT1}\selectfont\endgroup}
\usepackage{xcolor}
\usepackage[colorlinks=true,allcolors=blue,bookmarksnumbered=true]{hyperref}
\hypersetup{pdftitle={Averaged derivatives and reverse inequalities in quasiregular Bergman spaces},pdfauthor={Suman Das},pdfsubject={Weighted quasiregular Bergman spaces, concentration, and geometric embeddings},pdfkeywords={Quasiregular mapping, weighted Bergman space, averaged derivative, dominating set, Carleson measure, conformal mapping, compact embedding}}
\numberwithin{equation}{section}
\newtheorem{theorem}{Theorem}[section]
\newtheorem{proposition}[theorem]{Proposition}
\newtheorem{lemma}[theorem]{Lemma}
\newtheorem{corollary}[theorem]{Corollary}
\theoremstyle{definition}

\theoremstyle{remark}

\newtheorem*{remark*}{Remark}

\newcommand{\R}{\mathbb R}

\newcommand{\A}{\mathcal A}
\newcommand{\dd}{\,\mathrm d}

\newcommand{\diam}{\operatorname{diam}}
\newcommand{\osc}{\operatorname{osc}}
\newcommand{\cof}{\operatorname{cof}}
\newcommand{\supp}{\operatorname{supp}}
\newcommand{\avint}{\fint}
\newcommand{\Dav}{D_{\mathrm{av}}}
\newcommand{\Kstar}{K_*(n)}
\DeclareMathOperator{\arctanh}{arctanh}
\allowdisplaybreaks[2]
\title[Quasiregular Bergman spaces]{Averaged derivatives and reverse inequalities in quasiregular Bergman spaces}
\author{Suman Das}
\address{Department of Mathematics, Shantou University, Shantou, Guangdong 515063, P.~R. China}
\email{suman@stu.edu.cn}
\date{}
\begin{document}
\begin{abstract}
We study weighted Bergman spaces of quasiregular mappings on the unit ball in $\R^n$, with weight $(1-|x|)^\alpha$, $\alpha>-1$. For every fixed distortion bound $K\ge1$ and every $p>0$, we characterize Bergman integrability by averaged derivatives. We also give geometric conditions under which integration over a subset, or against another measure, controls the full Bergman integral. When $K$ is sufficiently large, the dominating sets are precisely the relatively dense sets. For comparison, we determine the dominating sets for the M\"obius class $K=1$ in dimensions $n\ge3$. At the critical exponent $p=n+\alpha$, its domination criterion differs from that for every fixed $K>1$.

Our results extend aspects of the Bergman theory developed by Nolder~\cite{Nolder90} and by Oleinik and Vuorinen~\cite{OV}, and complement the recent work of Adamowicz and Gonz\'alez~\cite{AG21,AG25} and Adamowicz and Caama\~no~\cite{AC} on quasiregular Hardy spaces.
\end{abstract}
\maketitle

\section{Introduction}\label{sec:intro}

Let $B=B(0,1)$ be the unit ball in $\R^n$, $n\ge2$, and let $\delta(x)=1-|x|$.
For $K\ge1$, a continuous mapping $f\in W^{1,n}_{\mathrm{loc}}(B,\R^n)$ is
$K$-quasiregular if
\begin{equation}\label{eq:distortion}
 |Df(x)|^n\le KJ_f(x)\quad\text{for almost every }x\in B.
\end{equation}
Here $|Df|$ is the operator norm and $J_f=\det Df$. We use the
outer-distortion convention and include constant mappings. For $\alpha>-1$,
we write
\[
 dV_\alpha(x)=\delta(x)^\alpha\dd x,\qquad
 V_\alpha(E)=\int_EdV_\alpha,\qquad Q=n+\alpha.
\]
The weighted Bergman class $\A^p_{K,\alpha}$ consists of the
$K$-quasiregular mappings on $B$ for which
\[
 \|f\|_{p,\alpha}
 =\left(\int_B|f(x)|^p\,dV_\alpha(x)\right)^{1/p}<\infty,
 \qquad 0<p<\infty.
\]
This class is generally nonlinear; we use the metric inherited from
$L^p(dV_\alpha)$, described in Section~\ref{sec:metric}. We write $A\asymp B$ when each quantity is bounded by a constant multiple of the other.

Nolder~\cite{Nolder90} established Bergman measure inequalities for
quasiregular mappings, and Oleinik and Vuorinen~\cite{OV} developed
weighted embeddings and compactness on quasiconformal balls. Bergman
estimates for the harmonic quasiregular subclass are studied in~\cite{DR}.

Our primary motivation comes from the work of Adamowicz and
Gonz\'alez~\cite{AG21,AG25} on quasiregular Hardy spaces, which extends
the quasiconformal theory of Astala and Koskela~\cite{AK}.
In~\cite{AG25}, they relate Hardy integrability to boundary values and
nontangential maximal functions for nonvanishing quasiregular mappings
satisfying certain growth and multiplicity conditions. They also establish a
Hardy--Bergman embedding under these hypotheses. The Bergman integral
offers a complementary point of view: \textit{since it is taken over the ball,
one can start from interior estimates without first constructing a
boundary-value function.} We study this integral on the full class with
a fixed distortion bound. One naturally asks whether it can be described by averaged
derivatives, and when integration over a smaller set, or against another
measure, controls its size.

In~\cite{Nolder}, Nolder used averaged derivatives in the study of continuity
and BMO. Recently, Adamowicz and Caama\~no~\cite{AC} studied them in the Hardy
space setting and compared their integrals at different fixed averaging radii.

For $0<c<1$, define
\[
 D_{\mathrm{av},c}f(x)
 =\left(\avint_{B(x,c\delta(x))}|Df(y)|^n\dd y\right)^{1/n},
 \qquad \Dav f=D_{\mathrm{av},1/2}f,
\]
where $\avint$ denotes the ordinary volume average. Averaging is needed
because the derivative of a bounded quasiregular mapping need not be
locally $L^p$-integrable for large $p$.

Our first main theorem characterizes weighted Bergman integrability in terms
of these averaged derivatives.

\begin{theorem}\label{thm:mainnorm}
Let $K\ge1$, $0<p<\infty$, $\alpha>-1$, and $0<c<1$. Every
$K$-quasiregular mapping $f:B\to\R^n$ satisfies
\begin{equation}\label{eq:mainnorm}
 \|f\|_{p,\alpha}^p
 \asymp |f(0)|^p+
 \int_B\bigl(\delta(x)D_{\mathrm{av},c}f(x)\bigr)^p\,dV_\alpha(x).
\end{equation}
Equivalently,
\[
 \|f\|_{p,\alpha}^p
 \asymp |f(0)|^p+
 \int_B\left(\int_{B(x,c\delta(x))}J_f(y)\dd y\right)^{p/n}
 dV_\alpha(x).
\]
The constants depend only on $n,K,p,\alpha,c$. In particular, $f\in\A^p_{K,\alpha}$ if and only if either of
the above integrals is finite.
\end{theorem}

The value at zero accounts for constant mappings. The characterization
holds for all $p>0$, without additional growth or multiplicity assumptions.

We next ask when a part of the ball determines the whole integral.
A measurable set $E\subset B$ is \emph{dominating} for
$\A^p_{K,\alpha}$ if
\begin{equation}\label{eq:domination}
 \|f\|_{p,\alpha}^p\le C_E\int_E|f|^p\,dV_\alpha
 \qquad(f\in\A^p_{K,\alpha}).
\end{equation}
For $a\in B$ and $0<r<1$, let
\[
 \Delta_r(a)=\{x\in B:\rho(a,x)<r\},\qquad
 \rho(a,x)=\frac{|x-a|}
 {\bigl(|x-a|^2+(1-|a|^2)(1-|x|^2)\bigr)^{1/2}}.
\]
These are hyperbolic balls, with $V_\alpha(\Delta_r(a))\asymp\delta(a)^Q$
for fixed $r$. We call $E$ \emph{relatively dense} if some $r\in(0,1)$
and $\eta>0$ satisfy
\begin{equation}\label{eq:density}
 V_\alpha(E\cap\Delta_r(a))\ge\eta V_\alpha(\Delta_r(a))
 \qquad(a\in B).
\end{equation}
This is the familiar geometric condition in the analytic theory of
dominating sets~\cite{Lue84,Lue85,DS,GW}.

The next theorem shows that relative density is sufficient for domination.
It also gives two ranges in which this condition is necessary.

\begin{theorem}\label{thm:maindomination}
Let $K\ge1$, $0<p<\infty$, and $\alpha>-1$. Every relatively dense
measurable set is dominating for $\A^p_{K,\alpha}$. There is a dimensional
constant $\Kstar$, with $K_*(2)=1$, such that the converse holds if
$K\ge\Kstar$. The converse also holds whenever
\begin{equation}\label{eq:radialrange}
 pK^{1/(n-1)}>n+\alpha.
\end{equation}
\end{theorem}

In particular, relative density characterizes dominating sets in the plane
for every $K\ge1$ and $p>0$. In higher dimensions, the same conclusion holds
for every $p>0$ once $K$ is sufficiently large. Section~\ref{sec:dominating}
also gives reverse inequalities for more general measures.

For comparison, we turn to $K=1$ in dimensions $n\ge3$. By Liouville's
theorem, this class consists of M\"obius transformations and constants.
To describe the critical case, we use the boundary-cap proportions
\[
 \Theta_E(\omega,r)
 =\frac{V_\alpha(E\cap B(\omega,r))}
        {V_\alpha(B\cap B(\omega,r))},
 \qquad \omega\in\partial B,\quad 0<r<1.
\]

The following is our final main theorem, which gives a complete description of dominating sets in this
class. It also identifies when the closed unit ball is compact.

\begin{samepage}
\begin{theorem}\label{thm:mainconformal}
Suppose $n\ge3$, $\alpha>-1$, and $0<p<\infty$. A measurable set
$E\subset B$ is dominating for $\A^p_{1,\alpha}$ precisely under the
following conditions:
\begin{enumerate}
\item if $p<Q$, then $V_\alpha(E)>0$;
\item if $p=Q$, then, for one fixed $r_0\in(0,1)$,
\begin{equation}\label{eq:logdensity}
 \liminf_{d\downarrow0}\ \inf_{\omega\in\partial B}
 \frac{1}{\log(r_0/d)}
 \int_d^{r_0}\Theta_E(\omega,r)\,\frac{dr}{r}>0;
\end{equation}
\item if $p>Q$, then $E$ is relatively dense.
\end{enumerate}
Condition~\eqref{eq:logdensity} is independent of $r_0$. The closed unit
ball of $\A^p_{1,\alpha}$ is compact in its ambient $L^p_\alpha$ metric
if and only if $p<Q$.
\end{theorem}
\end{samepage}

At the critical exponent, condition~\eqref{eq:logdensity} is weaker than
relative density and permits arbitrarily large gaps in hyperbolic distance.
For every fixed $K>1$, however, relative density again characterizes
domination at $p=Q$; see Corollary~\ref{cor:criticaldistortion}.

The rest of the paper is organised as follows. Section~\ref{sec:local} collects the local estimates and geometric facts
used in the proofs. The averaged-derivative theorem is proved in
Section~\ref{sec:norms}. Sections~\ref{sec:concentration}--\ref{sec:distortion}
treat concentration, reverse inequalities, and the conformal class.
The remaining sections give the related forward embeddings and metric
properties. The Zorich parametrization used to construct the poles is
included in Appendix~\ref{app:zorich}.

\section{Local estimates and the geometry of the ball}\label{sec:local}

Throughout the paper, $C$ denotes a positive constant whose value may change from line to line. Its dependence is indicated in each statement. The oscillation of a vector-valued mapping on $E$ is
\[
 \osc_E f=\sup\{|f(x)-f(y)|:x,y\in E\}.
\]
We will repeatedly use the elementary inequality
\begin{equation}\label{eq:elementarypower}
 |u+v|^s\le c_s(|u|^s+|v|^s),\qquad
 c_s=\max\{1,2^{s-1}\},\qquad s>0.
\end{equation}

\subsection{Interior estimates}

We begin with the local supremum and H\"older estimates.

\begin{lemma}\label{lem:local}
Let $f$ be $K$-quasiregular on a domain $\Omega\subset\R^n$. For each $s>0$ and every ball $B(a,2r)\Subset\Omega$,
\begin{equation}\label{eq:localsup}
 \sup_{B(a,r)}|f|
 \le C\left(\avint_{B(a,2r)}|f|^s\dd x\right)^{1/s}.
\end{equation}
There are $\vartheta=\vartheta(n,K)>0$ and $C=C(n,K,s)$ such that
\begin{equation}\label{eq:localholder}
 |f(x)-f(y)|
 \le C\left(\frac{|x-y|}{r}\right)^\vartheta
 \left(\avint_{B(a,2r)}|f|^s\dd z\right)^{1/s},
 \qquad x,y\in B(a,r/2).
\end{equation}
The same statements hold after subtracting a constant vector from $f$.
\end{lemma}

\begin{proof}
The coordinate functions of a quasiregular mapping solve $\mathcal A$-harmonic equations of type $n$, with structure constants controlled by $n$ and $K$. The supremum estimate follows from the scalar estimate discussed after Lemma~2.1 in~\cite{AG25}; see also~\cite[Theorem~3.34 and Chapter~14]{HKM}. Applying that scalar estimate to each coordinate and using equivalence of the Euclidean and coordinate norms in $\R^n$, we obtain \eqref{eq:localsup}.

For $x,y\in B(a,r/2)$, the local H\"older estimate bounds $|f(x)-f(y)|$ by $C(|x-y|/r)^\vartheta$ times the supremum of $|f|$ on an intermediate ball. An application of \eqref{eq:localsup} on a slightly larger ball then yields \eqref{eq:localholder}. This is the usual local H\"older theorem for quasiregular mappings; see~\cite[Chapter~6]{HKM} and~\cite[Section~3]{FN}. Finally, subtracting a constant leaves both sides of the distortion inequality unchanged.
\end{proof}

We use two further standard structural facts. A nonconstant quasiregular mapping is open and discrete; in particular, its inverse image of a point is a discrete set. A locally uniform finite limit of mappings satisfying \eqref{eq:distortion} satisfies the same inequality and belongs to $W^{1,n}_{\mathrm{loc}}$. The limit is either nonconstant quasiregular or constant. These are the open-and-discrete theorem and the compactness theorem of quasiregular mapping theory; see~\cite{Rickman}, in particular Theorem~VI.8.6 for the latter. It follows from these facts, Lemma~\ref{lem:local}, and the Arzel\`a--Ascoli theorem that a locally bounded fixed-$K$ family has locally uniformly convergent subsequences. All subsequent compactness arguments use this fixed distortion bound.

For the derivative estimates it is useful to have the following direct energy inequality.

\begin{lemma}\label{lem:caccioppoli}
If $f$ is $K$-quasiregular on $\Omega$, $b\in\R^n$, and $0\le\eta\in C_c^\infty(\Omega)$, then
\begin{equation}\label{eq:caccioppoli}
 \int_\Omega\eta^n|Df|^n\dd x
 \le (nK)^n\int_\Omega|f-b|^n|\nabla\eta|^n\dd x.
\end{equation}
\end{lemma}

\begin{proof}
Let $C_{1j}$ denote the entries in the first row of $\cof Df$. By the cofactor identity,
\[
 \sum_{j=1}^n\partial_j\bigl((f_1-b_1)C_{1j}\bigr)=J_f
\]
in distributions. For a smooth mapping, this follows by differentiating the determinant and using $\sum_j\partial_jC_{1j}=0$. For the present mapping, mollification on a neighborhood of $\supp\eta$ proves the identity: the derivatives converge in $L^n$, the cofactors in $L^{n/(n-1)}$, and the Jacobians in $L^1$. Since $f$ is continuous, its mollifications also converge uniformly on that neighborhood. We may therefore pass to the limit in the product with the cofactors.

Set $I=\int\eta^n|Df|^n$. Testing the identity with $\eta^n$ and using \eqref{eq:distortion}, we obtain
\[
 I\le K\int\eta^nJ_f
 =-nK\int\eta^{n-1}(f_1-b_1)
            \sum_j C_{1j}\partial_j\eta.
\]
The Euclidean norm of a row of $\cof Df$ is at most $|Df|^{n-1}$. Hence H\"older's inequality yields
\[
 I\le nK I^{(n-1)/n}
       \left(\int|f-b|^n|\nabla\eta|^n\right)^{1/n}.
\]
If $I=0$ there is nothing to prove. Otherwise, we divide by $I^{(n-1)/n}$ and take the $n$th power to obtain \eqref{eq:caccioppoli}.
\end{proof}

\subsection{Weighted volume near the boundary}

The exponent $Q=n+\alpha$ describes the weighted volume of a boundary cap.

\begin{lemma}\label{lem:cap}
For $\omega\in\partial B$ and $0<R\le1$,
\[
 V_\alpha(B\cap B(\omega,R))\asymp R^Q.
\]
Moreover, for every fixed $L\ge1$ there is $C_L$ such that, if $a\in B$, $t=\delta(a)$, and $0<\varepsilon<1/2$, then
\begin{equation}\label{eq:thincap}
 V_\alpha\bigl(\{x\in B:|x-a|<Lt,\ \delta(x)<\varepsilon t\}\bigr)
 \le C_L\varepsilon^{\alpha+1}t^Q.
\end{equation}
The constants depend only on the displayed parameters and on $n,\alpha$.
\end{lemma}

\begin{proof}
Consider first $R<1/4$. If $x=(1-u)\zeta\in B(\omega,R)$, where $\zeta\in\partial B$, then $0<u<R$ and $|\zeta-\omega|<2R$. The area of this spherical cap is at most $CR^{n-1}$. Using polar coordinates, we obtain the upper estimate
\[
 V_\alpha(B\cap B(\omega,R))
 \le CR^{n-1}\int_0^R u^\alpha\dd u
 =\frac{C}{\alpha+1}R^Q.
\]
For the lower estimate, the ball centered at $(1-R/2)\omega$ with radius $R/8$ lies in the cap, and $\delta(x)\asymp R$ on that ball. Its weighted volume is comparable to $R^Q$. The range $1/4\le R\le1$ follows from this fixed lower bound and $V_\alpha(B)<\infty$.

To prove \eqref{eq:thincap}, first suppose $Lt<1/4$. Then $|a|>3/4$. On the indicated set, $|x|>1/2$, and the possible directions $x/|x|$ form a spherical cap of radius at most $C Lt$. Its area is at most $C(Lt)^{n-1}$. The claimed bound now follows by integrating through the depths $0<u<\varepsilon t$. If $Lt\ge1/4$, we use the full boundary collar instead. Its weighted volume is at most $C(\varepsilon t)^{\alpha+1}$, and $t^{-(n-1)}\le(4L)^{n-1}$ absorbs the missing factor. This proves \eqref{eq:thincap} in both cases.
\end{proof}

\subsection{Hyperbolic balls}

We give the geometric details used to transfer local statements from a fixed ball to a ball near the boundary. For $a\in B$, define
\[
 \phi_a(u)=
 \frac{(1-|a|^2)(u+a)+|u+a|^2a}
 {1+2\langle a,u\rangle+|a|^2|u|^2}.
\]
This is an orientation-preserving conformal automorphism of $B$, with $\phi_a(0)=a$ and inverse $\phi_{-a}$; for $a=0$ it is the identity. These standard ball transformations are discussed in~\cite{Ahlfors,AG25}. Direct calculation gives
\begin{align}
 1-|\phi_a(u)|^2
 &=\frac{(1-|a|^2)(1-|u|^2)}
 {1+2\langle a,u\rangle+|a|^2|u|^2},\label{eq:automorphismdepth}\\
 J_{\phi_a}(u)
 &=\left(\frac{1-|a|^2}
 {1+2\langle a,u\rangle+|a|^2|u|^2}\right)^n.
 \label{eq:automorphismjac}
\end{align}
Also $|\phi_{-a}(x)|=\rho(a,x)$, so $\Delta_r(a)=\phi_a(B(0,r))$.

\begin{lemma}\label{lem:hyperbolic}
For each fixed $R<1$,
\begin{equation}\label{eq:weighttransport}
 \delta(\phi_a(u))\asymp\delta(a),\qquad
 \delta(\phi_a(u))^\alpha J_{\phi_a}(u)\asymp\delta(a)^Q,
 \qquad |u|\le R,
\end{equation}
uniformly in $a\in B$. In particular, for fixed $r\in(0,1)$,
\begin{equation}\label{eq:hyperbolicvolume}
 V_\alpha(\Delta_r(a))\asymp\delta(a)^Q,
 \qquad \diam\Delta_r(a)\le C_r\delta(a).
\end{equation}
If $0<r<R<1$, there is a sequence $(a_j)$ such that the balls $\Delta_r(a_j)$ cover $B$ and the balls $\Delta_R(a_j)$ have bounded overlap.
\end{lemma}

\begin{proof}
For $|u|\le R$, the denominator in \eqref{eq:automorphismdepth} and \eqref{eq:automorphismjac} lies between $(1-R)^2$ and $(1+R)^2$. Since $1-|a|^2\asymp\delta(a)$ and $1-|\phi_a(u)|^2\asymp\delta(\phi_a(u))$, it follows from these identities that \eqref{eq:weighttransport} holds. Integrating, we obtain the volume assertion. The diameter assertion also follows from the derivative bound in \eqref{eq:automorphismjac}; alternatively, the Euclidean center and radius of $\Delta_r(a)$ are
\[
 \frac{(1-r^2)a}{1-r^2|a|^2},\qquad
 \frac{r(1-|a|^2)}{1-r^2|a|^2}.
\]

For the covering, use the hyperbolic metric $d(a,x)=2\arctanh\rho(a,x)$. Choose a maximal $h$-separated sequence, with $0<h<2\arctanh r$. Maximality gives a covering by the $d$-balls of radius $h$, hence by $\Delta_r(a_j)$. The $d$-balls of radius $h/2$ centered at $a_j$ are disjoint. The volume element
\[
 d\tau(x)=\frac{\dd x}{(1-|x|^2)^n}
\]
is invariant under the ball automorphisms, by \eqref{eq:automorphismdepth} and \eqref{eq:automorphismjac}. Thus every $d$-ball of a fixed finite radius has the same positive finite $\tau$-volume. If a point belongs to several $\Delta_R(a_j)$, the disjoint balls of radius $h/2$ associated with their centers all lie in the ball about that point of radius $2\arctanh R+h/2$. Comparing invariant volumes bounds their number.
\end{proof}

An immediate consequence, used later, is the uniform local estimate
\begin{equation}\label{eq:hyperbolicsup}
 \sup_{\Delta_r(a)}|f|^p
 \le C\delta(a)^{-Q}
       \int_{\Delta_R(a)}|f|^p\,dV_\alpha,
 \qquad 0<r<R<1.
\end{equation}
Indeed, apply Lemma~\ref{lem:local} to $f\circ\phi_a$ on a finite covering of $\overline{B(0,r)}$ by balls compactly contained in $B(0,R)$, and then use \eqref{eq:weighttransport}. Precomposition by $\phi_a$ preserves the outer-distortion bound.

\subsection{Point evaluation and convergence}\label{sec:metric}

The following point-evaluation estimates follow from the local mean inequality; see Nolder~\cite[Corollary~1.6]{Nolder90} and Oleinik--Vuorinen~\cite[Corollary~6.2]{OV}.

\begin{proposition}\label{prop:evaluation}
For $f\in\A^p_{K,\alpha}$ and $a\in B$,
\begin{equation}\label{eq:localevaluation}
 |f(a)|^p\le C\delta(a)^{-Q}
 \int_{B(a,\delta(a)/2)}|f(x)|^p\,dV_\alpha(x).
\end{equation}
Consequently,
\begin{equation}\label{eq:evaluationgrowth}
 |f(a)|\le C\delta(a)^{-Q/p}\|f\|_{p,\alpha}.
\end{equation}
For each fixed $f\in\A^p_{K,\alpha}$,
\begin{equation}\label{eq:littleoh}
 \delta(a)^{Q/p}|f(a)|\longrightarrow0
 \quad\text{as }|a|\to1.
\end{equation}
\end{proposition}

\begin{proof}
On $B(a,\delta(a)/2)$ one has $\delta(x)\asymp\delta(a)$. Apply \eqref{eq:localsup} with outer radius $\delta(a)/2$ and exponent $p$. On replacing ordinary volume by $dV_\alpha$, we obtain \eqref{eq:localevaluation}. Enlarging the integral to $B$, we then deduce \eqref{eq:evaluationgrowth}.

The ball in \eqref{eq:localevaluation} lies in $\{x:\delta(x)<3\delta(a)/2\}$. The integral of $|f|^p$ over this boundary collar tends to zero as $\delta(a)\to0$, by integrability. This proves \eqref{eq:littleoh} for all approaches to the boundary.
\end{proof}

For later use, denote the size of point evaluation by
\[
 \mathcal E_{K,p,\alpha}(a)
 =\sup\{|f(a)|:f\in\A^p_{K,\alpha},\ \|f\|_{p,\alpha}\le1\}.
\]
Thus \eqref{eq:evaluationgrowth} gives
$\mathcal E_{K,p,\alpha}(a)\le C\delta(a)^{-Q/p}$.

We equip $\A^p_{K,\alpha}$ with the metric
\begin{equation}\label{eq:metric}
 d_p(f,g)=
 \begin{cases}
  \|f-g\|_{p,\alpha},&p\ge1,\\[1mm]
  \displaystyle\int_B|f-g|^p\,dV_\alpha,&0<p<1.
 \end{cases}
\end{equation}
For $p<1$, the triangle inequality follows from \eqref{eq:elementarypower} with $c_p=1$. Since our mappings are continuous, distance zero means equality everywhere. Thus no choice of representatives is needed inside the quasiregular class.

\begin{proposition}\label{thm:complete}
The metric space $(\A^p_{K,\alpha},d_p)$ is complete. If $f_j,f\in\A^p_{K,\alpha}$ and $d_p(f_j,f)\to0$, then $f_j\to f$ locally uniformly on $B$.
\end{proposition}

\begin{proof}
We first prove the convergence assertion directly. In fact, only continuity of the proposed limit is needed. Let $E\Subset B$, and choose a compact neighborhood $E'\Subset B$ of $E$. For sufficiently small fixed $r>0$, every $B(a,r)$ with $a\in E$ is contained in $E'$. Define
\[
 \omega_f(r)=\sup\{|f(x)-f(y)|:x,y\in E',\ |x-y|\le r\}.
\]
The difference $f_j-f$ need not be quasiregular. Translation by the fixed vector $f(a)$ does preserve quasiregularity, so we use $f_j-f(a)$ instead. Applying \eqref{eq:localsup} to it and then using \eqref{eq:elementarypower}, we obtain
\begin{align*}
 |f_j(a)-f(a)|^p
 &\le Cr^{-n}\int_{B(a,r)}|f_j(x)-f(a)|^p\dd x\\
 &\le Cr^{-n}\|f_j-f\|_{p,\alpha}^p+C\omega_f(r)^p.
\end{align*}
The weight is bounded above and below on $E'$, and the constants are independent of $a$ and $j$. First let $j\to\infty$, and then let $r\downarrow0$. Uniform continuity on $E'$ shows that the convergence is uniform on $E$.

Now let $(f_j)$ be a Cauchy sequence in \eqref{eq:metric}. Since the ambient weighted $L^p$ space is complete, the sequence has a limit $g\in L^p_\alpha(B,\R^n)$. In particular, the sizes $\|f_j\|_{p,\alpha}$ are bounded. By Proposition~\ref{prop:evaluation} and Lemma~\ref{lem:local}, the sequence is locally uniformly bounded and equicontinuous. A diagonal Arzel\`a--Ascoli argument then allows us to pass to a subsequence that converges locally uniformly to a continuous mapping $h$. The quasiregular compactness theorem recalled in Section~\ref{sec:local} shows that $h$ satisfies the original distortion bound, with constants allowed.

The same subsequence converges to $g$ in $L^p_\alpha$. On each compact ball its uniform limit and its $L^p$ limit agree almost everywhere. Exhausting $B$, we obtain $h=g$ almost everywhere. Hence $h\in\A^p_{K,\alpha}$, and the whole sequence converges to $h$ in the metric. This proves completeness.
\end{proof}

Both the uniform distortion bound and the inclusion of constants are relevant. The mappings $f_j(x)=x/j$ converge to a constant. If the distortion is not fixed, the quasiconformal linear mappings
\[
 f_j(x)=(x_1,j^{-1}x_2,\ldots,j^{-1}x_n)
\]
converge uniformly to a nonconstant rank-one mapping. The latter has zero Jacobian and nonzero derivative, so it cannot satisfy \eqref{eq:distortion} with any finite $K$.
\section{Averaged derivatives and Jacobians}\label{sec:norms}

For $0<c<1$, define the local maximal function
\[
 \mathcal M_cf(x)=\sup_{y\in B(x,c\delta(x))}|f(y)|.
\]
This local maximal function has the same weighted integrability as $f$.

\begin{lemma}\label{thm:maximal}
For every fixed $c\in(0,1)$ and every $K$-quasiregular mapping $f$,
\[
 \int_B(\mathcal M_cf)^p\,dV_\alpha
 \asymp\int_B|f|^p\,dV_\alpha.
\]
The constants depend only on $n,K,p,\alpha,c$.
\end{lemma}

\begin{proof}
The lower estimate follows from $\mathcal M_cf\ge|f|$. Choose $c<c'<1$. Cover $B(x,c\delta(x))$ by a bounded number of balls whose doubled balls lie in $B(x,c'\delta(x))$. Their radii are fixed multiples of $\delta(x)$. Lemma~\ref{lem:local} and comparability of the weight imply
\[
 (\mathcal M_cf(x))^p
 \le C\delta(x)^{-Q}
       \int_{B(x,c'\delta(x))}|f(y)|^p\,dV_\alpha(y).
\]
Integrating in $x$ and applying Tonelli's theorem, we see that it remains to prove
\[
 \int_{\{x:|x-y|<c'\delta(x)\}}\delta(x)^{-n}\dd x\le C.
\]
Indeed, since $\delta$ is Lipschitz, we have on this set
\[
 (1-c')\delta(x)<\delta(y)<(1+c')\delta(x).
\]
The set is therefore contained in a Euclidean ball about $y$ of radius $C\delta(y)$, and the integrand is at most $C\delta(y)^{-n}$. This proves the bound. Tonelli's theorem also shows that the argument is valid with infinite integrals.
\end{proof}

We are now ready to prove the averaged-derivative characterization (Theorem~\ref{thm:mainnorm}).

\begin{proof}[Proof of Theorem~\ref{thm:mainnorm}]
Fix $c\in(0,1)$ and put
\[
 G(x)=\delta(x)D_{\mathrm{av},c}f(x).
\]
All constants in this proof may depend on $c$, in addition to
$n,K,p,\alpha$. We first show that the Bergman integral controls the
right-hand side of \eqref{eq:mainnorm}. Choose $c'=(1+c)/2$ and a cutoff
which equals one on $B(x,c\delta(x))$, vanishes outside
$B(x,c'\delta(x))$, and has gradient bounded by $C_c/\delta(x)$.
Applying the Caccioppoli inequality with the constant vector equal to zero, we obtain
\[
 G(x)\le C\mathcal M_{c'}f(x).
\]
It now follows from the local maximal estimate and bounded evaluation at zero that
\[
 |f(0)|^p+\int_B G(x)^p\,dV_\alpha(x)
 \le C\|f\|_{p,\alpha}^p.
\]
This assertion is automatic if the right-hand side is infinite.

For the reverse inequality, suppose that $G\in L^p_\alpha(B)$.
The local supremum estimate first converts energy into oscillation.
Indeed, let $b_x$ be the average of $f$ on $B(x,c\delta(x))$.
Translation by $b_x$ preserves quasiregularity. Applying the local estimate
with exponent $n$, followed by the Poincar\'e inequality on this ball,
we obtain
\begin{equation}\label{eq:oscillationenergy}
 \begin{split}
 \osc_{B(x,c\delta(x)/2)}f
 &\le C\left(\avint_{B(x,c\delta(x))}|f-b_x|^n\dd y\right)^{1/n}\\
 &\le Cc\delta(x)
       \left(\avint_{B(x,c\delta(x))}|Df|^n\dd y\right)^{1/n}
 \le CG(x).
 \end{split}
\end{equation}
Poincar\'e's inequality is applied componentwise. Replacing the Euclidean
matrix norm by the operator norm changes only a dimensional constant.

We first use \eqref{eq:oscillationenergy} to control the fixed interior
ball. Set $h=c/16$, and let $u,v\in B(0,1/2)$ satisfy $|u-v|\le h$.
For every $z\in B(u,h)$, we have $z\in B(0,3/4)$ and hence
$\delta(z)>1/4$. Therefore
\[
 |u-z|<h,\qquad
 |v-z|<2h=\frac c8<\frac{c\delta(z)}2.
\]
Both points lie in the oscillation ball centered at $z$. Thus, by
\eqref{eq:oscillationenergy},
\[
 |f(u)-f(v)|^p\le CG(z)^p
 \qquad(z\in B(u,h)).
\]
The weight is bounded above and below by positive constants on
$B(0,3/4)$, and consequently
$V_\alpha(B(u,h))\asymp h^n$, uniformly in $u$.
Integrating the preceding inequality over $B(u,h)$ and taking
$p$th roots yields
\[
 |f(u)-f(v)|\le C\|G\|_{p,\alpha},
\]
where the dependence on $h$ is absorbed into the allowed dependence
on $c$.

Now divide the segment from zero to any $x\in B(0,1/2)$ into
pieces of length at most $h$. The number of pieces is bounded
in terms of $c$, and all subdivision points remain in $B(0,1/2)$.
Adding the successive differences, we obtain
\begin{equation}\label{eq:interiorcontrol}
 \sup_{B(0,1/2)}|f|
 \le |f(0)|+C\|G\|_{p,\alpha}.
\end{equation}
This last step uses only the ordinary triangle inequality in
$\R^n$, so the argument also applies when $p<1$.

To reach a point near the boundary, we use a radial chain whose step
length fits the oscillation balls in \eqref{eq:oscillationenergy}.
Set $\lambda=1+c/4$. For $x=r\omega$, where $1/2<r<1$ and
$\omega\in\partial B$, let
\[
 x_j=\bigl(1-\lambda^j(1-r)\bigr)\omega,
\]
and stop at the first index $J$ for which
$\lambda^J(1-r)\ge1/2$. Then $x_j\in B\setminus\overline{B(0,1/2)}$
for $j<J$, whereas
\[
 1-\lambda/2<|x_J|\le1/2.
\]
In particular, the last point lies in the interior ball and the chain
does not pass through the origin. Moreover,
\[
 |x_{j+1}-x_j|=(\lambda-1)\delta(x_j)
             =c\delta(x_j)/4<c\delta(x_j)/2.
\]
Applying \eqref{eq:oscillationenergy} with center $x_j$ and adding the
successive differences gives
\begin{equation}\label{eq:radialchain}
 |f(x)|\le \sup_{B(0,1/2)}|f|
                +C\sum_{j=0}^{J-1}G(x_j).
\end{equation}
If the last point lies on the sphere of radius $1/2$, the same bound
follows by continuity.

For $j\ge0$, define $H_j(x)=G(x_j)$ on
$\{x:\lambda^j\delta(x)<1/2\}$, and put $H_j(x)=0$ elsewhere.
The points $x_j$ used in this definition all have modulus greater than
$1/2$. Under the radial change of variables $y=x_j$, the angular
variable is unchanged and
\[
 \delta(x)=\lambda^{-j}\delta(y),\qquad
 d|x|=\lambda^{-j}d|y|,
 \qquad \frac{|x|^{n-1}}{|y|^{n-1}}\le2^{n-1}.
\]
Consequently, integration in polar coordinates leads to
\begin{equation}\label{eq:chainintegral}
 \int_B H_j(x)^p\,dV_\alpha(x)
 \le 2^{n-1}\lambda^{-j(\alpha+1)}
           \int_B G(y)^p\,dV_\alpha(y).
\end{equation}
The decay comes from the radial variable and the weight; its exponent
is $\alpha+1$.

If $0<p\le1$, it follows from subadditivity, \eqref{eq:chainintegral}, and Tonelli's theorem that
\[
 \int_B\left(\sum_{j\ge0}H_j\right)^p\,dV_\alpha
 \le 2^{n-1}\sum_{j\ge0}\lambda^{-j(\alpha+1)}
                     \|G\|_{p,\alpha}^p.
\]
If $p>1$, we apply Minkowski's inequality to a finite sum and obtain
\[
 \left\|\sum_{j=0}^mH_j\right\|_{p,\alpha}
 \le 2^{(n-1)/p}\sum_{j=0}^m
          \lambda^{-j(\alpha+1)/p}\|G\|_{p,\alpha}.
\]
Monotone convergence allows $m$ to tend to infinity. Both geometric
series converge because $\alpha>-1$. These estimates, together with
\eqref{eq:radialchain}, \eqref{eq:interiorcontrol}, and the finiteness of
$V_\alpha(B)$, imply
\[
 \|f\|_{p,\alpha}^p
 \le C\left(|f(0)|^p+\int_BG^p\,dV_\alpha\right).
\]
This proves the equivalence, including its assertion about finiteness.

Finally, quasiregularity and the singular-value inequality give
$J_f\le|Df|^n\le KJ_f$ almost everywhere. Since
$|B(x,c\delta(x))|=|B(0,1)|c^n\delta(x)^n$, one has
\[
 G(x)=\frac1{c|B(0,1)|^{1/n}}
       \left(\int_{B(x,c\delta(x))}|Df(y)|^n\dd y\right)^{1/n}.
\]
The asserted integrated-Jacobian formula follows on replacing $|Df|^n$ by $J_f$.
\end{proof}

The integrated Jacobian retains multiplicity. For a nonconstant quasiregular mapping, the area formula gives
\[
 \int_UJ_f(x)\dd x=\int_{\R^n}N(y,f,U)\dd y,
\]
where $N(y,f,U)$ counts preimages in $U$ almost everywhere; see~\cite{Rickman}. The Jacobian expression therefore measures local image volume with multiplicity.

The averaging in the theorem is important. For $0<a<1$, the radial mapping
\[
 f(x)=|x|^{a-1}x,\qquad f(0)=0,
\]
is bounded and has outer distortion $1/a$. Its radial singular value is $a|x|^{a-1}$ and its other singular values are $|x|^{a-1}$. Hence
\[
 \int_{B(0,1/2)}|Df|^p\dd x=\infty
 \quad\text{when }p\ge\frac{n}{1-a}.
\]
It follows that the unaveraged expression $\int_B\delta^p|Df|^p\,dV_\alpha$ cannot replace the derivative expression in an all-$p$ theorem.

Applying Theorem~\ref{thm:mainnorm} to $f-f(0)$, we obtain the centered form:
\[
 \int_B|f-f(0)|^p\,dV_\alpha
 \asymp\int_B\bigl(\delta D_{\mathrm{av},c}f\bigr)^p\,dV_\alpha.
\]
\begin{remark*}
The comparison of the derivative integrals for different fixed radii is the weighted radius-comparison theorem of Adamowicz and Caama\~no~\cite[Theorem~3.7]{AC}. In their notation $B_x=B(x,\delta(x)/2)$ and
\[
 a_{f,\lambda}(x)
 =\left(\frac{1}{|2B_x|}\int_{\lambda B_x}|Df|^n\,dy\right)^{1/n},
 \qquad a_{f,2c}=cD_{\mathrm{av},c}f.
\]
Take the weight there to be $u=\delta^{p+\alpha}$. It satisfies the required local comparability because $\delta(y)/\delta(x)$ lies between $1/2$ and $3/2$ on $B_x$. Their theorem therefore compares the derivative integrals above for any two fixed $c\in(0,1)$.
\end{remark*}

\section{Concentrated test mappings}\label{sec:concentration}

The converse results use mappings whose mass is concentrated near a prescribed point at the scale of its distance to the boundary. Such mappings can be obtained from the uniformly quasiregular power maps of Mayer~\cite[Theorem~2 and Section~5]{Mayer}; see~\cite[Section~3]{Nicks} for their all-dimensional formulation. We use an explicit construction through one inverse branch of a Zorich map. It allows every positive real pole order and gives the exact modulus needed for the estimates below. The underlying Zorich construction is included in Appendix~\ref{app:zorich}.

\begin{proposition}\label{thm:poles}
There is a constant $\Kstar\ge1$, with $K_*(2)=1$, such that the following holds. For every $\sigma>0$, $\omega\in\partial B$, and $t\ge0$, there is a $\Kstar$-quasiregular mapping $F_{\sigma,\omega,t}:B\to\R^n$ satisfying
\begin{equation}\label{eq:polemodulus}
 |F_{\sigma,\omega,t}(x)|=|x-(1+t)\omega|^{-\sigma}.
\end{equation}
The distortion bound is independent of $\sigma,\omega,t$.
\end{proposition}

\begin{proof}
In dimension two, the nonvanishing holomorphic function $z\mapsto z-(1+t)\omega$ has a holomorphic logarithm in the disk. Thus
\[
 F_{\sigma,\omega,t}(z)
 =\exp\bigl(-\sigma\log(z-(1+t)\omega)\bigr)
\]
has the required modulus, and $K_*(2)=1$.

For $n\ge3$, let $Z$ and $\Lambda$ be the maps in Lemma~\ref{lem:zorich}, proved in Appendix~\ref{app:zorich}. Put $y=(1+t)\omega$ and choose $R_\omega\in SO(n)$ with $R_\omega(-\omega)=e_n$. Since
\[
 \bigl(R_\omega(x-y)\bigr)_n=1+t-\langle x,\omega\rangle>0
 \qquad(x\in B),
\]
the inverse branch $\Lambda$ is defined at every $R_\omega(x-y)$, including when $t=0$. Set
\[
 A_\sigma=\sigma\operatorname{diag}(-1,1,\ldots,1,-1)
\]
and define
\begin{equation}\label{eq:poleconstruction}
 F_{\sigma,\omega,t}(x)
 =Z\bigl(A_\sigma\Lambda(R_\omega(x-y))\bigr).
\end{equation}
The map $A_\sigma$ is an orientation-preserving similarity: it reverses two coordinate signs and multiplies all coordinates by the same factor. Thus it has distortion one, independently of $\sigma$. By the chain rule, the outer distortion in \eqref{eq:poleconstruction} is at most the product of those of $Z$ and $\Lambda$. The inner maps are locally bi-Lipschitz, so the exceptional sets in the chain rule have null preimages. By Lemma~\ref{lem:zorich}, we may take
\[
 K_*(n)=\max\{3\pi/2,\,n-1+2\sqrt{n-1}\}^{4(n-1)}.
\]
Finally, the last coordinate of $A_\sigma\Lambda(R_\omega(x-y))$ is $-\sigma\log|x-y|$. Formula~\eqref{eq:zorichmodulus} now gives \eqref{eq:polemodulus}.
\end{proof}

\subsection{The size and concentration of the test mappings}

The following integral calculation determines the normalization.

\begin{lemma}\label{lem:poleintegral}
If $s>Q$, then
\begin{equation}\label{eq:poleintegral}
 \int_B|x-(1+t)\omega|^{-s}\,dV_\alpha(x)
 \asymp t^{Q-s},\qquad 0<t\le1,
\end{equation}
uniformly in $\omega\in\partial B$. For every $s>0$, a pole on the boundary satisfies
\begin{equation}\label{eq:boundarypoleintegral}
 \int_B|x-\omega|^{-s}\,dV_\alpha(x)<\infty
 \quad\Longleftrightarrow\quad s<Q.
\end{equation}
\end{lemma}

\begin{proof}
Put $y=(1+t)\omega$. Since $|x-y|\ge t$ for $x\in B$, decompose the ball into shells where
\[
 2^jt\le|x-y|<2^{j+1}t,\qquad j\ge0.
\]
The part of a ball $B(y,R)$ in $B$, for $R\ge t$, lies in $B(\omega,2R)$. By Lemma~\ref{lem:cap}, its weighted volume is at most $CR^Q$; for $R\ge1$, the same bound follows from $V_\alpha(B)<\infty$. Thus
\[
 \int_B|x-y|^{-s}\,dV_\alpha
 \le C\sum_{j\ge0}(2^jt)^{Q-s}
 \le Ct^{Q-s}.
\]
For the lower bound, use $B((1-t)\omega,t/8)$. On this ball $\delta(x)\asymp t$ and $|x-y|\asymp t$, so its contribution is at least $ct^{Q-s}$.

For a boundary pole, the same dyadic estimate at radii tending to zero proves finiteness when $s<Q$. To prove divergence when $s\ge Q$, put $t_j=4^{-j}$ and use the pairwise disjoint balls
\[
 B((1-t_j)\omega,t_j/8).
\]
Their contributions are bounded below by $ct_j^{Q-s}$. These terms do not form a summable sequence when $s\ge Q$, which proves \eqref{eq:boundarypoleintegral}.
\end{proof}

Fix $\sigma>Q/p$ for the remainder of this section. If $a\ne0$, write $a=(1-t)\omega$ with $t=\delta(a)$; at $a=0$ take $t=1$ and $\omega=e_1$. Define
\[
 h_a=\frac{F_{\sigma,\omega,t}}{\|F_{\sigma,\omega,t}\|_{p,\alpha}}.
\]
These are $\Kstar$-quasiregular mappings of norm one. By \eqref{eq:polemodulus} and \eqref{eq:poleintegral},
\begin{equation}\label{eq:peakprofile}
 |h_a(x)|^p\asymp
 t^{\sigma p-Q}|x-(1+t)\omega|^{-\sigma p}.
\end{equation}
In particular,
\begin{equation}\label{eq:peakbounds}
 |h_a(x)|^p\le Ct^{-Q}\quad(x\in B),\qquad
 |h_a(x)|^p\ge c_rt^{-Q}\quad(x\in\Delta_r(a)).
\end{equation}
For the first estimate use $|x-(1+t)\omega|\ge t$. For the second, use
\[
 |x-(1+t)\omega|\le|x-a|+2t\le(C_r+2)t
\]
and Lemma~\ref{lem:hyperbolic}. Moreover,
\begin{equation}\label{eq:peaklocalzero}
 h_a\longrightarrow0\quad\text{locally uniformly as }|a|\to1,
\end{equation}
because $\sigma p-Q>0$ and the exterior poles stay a positive distance from any fixed compact subset of $B$.

The lower bound on a fixed hyperbolic ball is enough for the necessary measure condition. For dominating sets we also need uniform control of the mass outside larger hyperbolic balls.

\begin{lemma}\label{lem:concentration}
The normalized mappings just defined satisfy
\begin{equation}\label{eq:uniformconcentration}
 \lim_{r\uparrow1}\sup_{a\in B}
       \int_{B\setminus\Delta_r(a)}|h_a|^p\,dV_\alpha=0.
\end{equation}
\end{lemma}

\begin{proof}
Set $s=\sigma p>Q$ and $t=\delta(a)$. Let $L>4$. If $|x-a|>Lt$, then
\[
 |x-(1+t)\omega|\ge|x-a|-2t>\tfrac12Lt.
\]
The shell calculation in Lemma~\ref{lem:poleintegral}, together with \eqref{eq:peakprofile}, gives
\begin{equation}\label{eq:peakfar}
 \int_{\{|x-a|>Lt\}}|h_a|^p\,dV_\alpha
 \le C L^{Q-s}.
\end{equation}
On the region $|x-a|\le Lt$, $\delta(x)<\varepsilon t$, we apply the first estimate in \eqref{eq:peakbounds} and \eqref{eq:thincap} to obtain
\begin{equation}\label{eq:peakthin}
 \int_{\{|x-a|\le Lt,\ \delta(x)<\varepsilon t\}}|h_a|^p\,dV_\alpha
 \le C_L\varepsilon^{\alpha+1}.
\end{equation}
The remaining points satisfy $|x-a|\le Lt$ and $\delta(x)\ge\varepsilon t$. Since $1-|a|^2\ge t$ and $1-|x|^2\ge\varepsilon t$, it follows from the definition of $\rho$ that
\[
 \rho(a,x)^2\le\frac{L^2}{L^2+\varepsilon}.
\]
Choose a fixed $r<1$ larger than the square root of the last expression. The remaining region lies in $\Delta_r(a)$, uniformly in $a$. First choose $L$ so large that \eqref{eq:peakfar} is small, and then choose $\varepsilon$ so small that \eqref{eq:peakthin} is small. This proves \eqref{eq:uniformconcentration}.
\end{proof}

\begin{corollary}\label{cor:sharpevaluation}
If $K\ge\Kstar$, then
\[
 \mathcal E_{K,p,\alpha}(a)\asymp\delta(a)^{-Q/p}.
\]
Also, for $0<p<q<\infty$,
\[
 \A^q_{K,\alpha}\subsetneq\A^p_{K,\alpha}.
\]
\end{corollary}

\begin{proof}
For the upper evaluation estimate, we use \eqref{eq:evaluationgrowth}. The lower estimate follows by evaluating $h_a$ at $a$ in \eqref{eq:peakbounds}. The inclusion follows from H\"older's inequality and finiteness of $V_\alpha(B)$. For strictness, choose $\sigma$ with $Q/q\le\sigma<Q/p$. The boundary-pole mapping $F_{\sigma,\omega,0}$ is in the $p$ class but not in the $q$ class, by \eqref{eq:boundarypoleintegral}.
\end{proof}

\subsection{A radial distortion range}

The preceding dimensional threshold works for every $p$ and $\alpha$. For specified exponents, the radial pole maps used in the earlier embedding theory~\cite[Sections~3 and~5]{OV} give another useful range. We record their outer distortion explicitly.

\begin{proposition}\label{prop:radialrange}
Suppose $K\ge1$ and \eqref{eq:radialrange} holds. The estimates
\eqref{eq:peakprofile}--\eqref{eq:uniformconcentration} then hold for
normalized $K$-quasiregular radial pole mappings, with $\sigma$ replaced
by some $\beta\in[1,K^{1/(n-1)}]$ satisfying $p\beta>Q$. In particular,
\[
 \mathcal E_{K,p,\alpha}(a)\asymp\delta(a)^{-Q/p}.
\]
\end{proposition}

\begin{proof}
Choose $\beta$ with $1\le\beta\le K^{1/(n-1)}$ and $p\beta>Q$. Let $U$ be an orientation-reversing orthogonal map, let $y=(1+t)\omega$, and define
\[
 F_t(x)=\frac{U(x-y)}{|x-y|^{1+\beta}}.
\]
Before the orthogonal map is applied, the radial derivative has eigenvalue $-\beta|x-y|^{-1-\beta}$ and the tangential derivatives have eigenvalue $|x-y|^{-1-\beta}$. Thus $U$ corrects the orientation and
\[
 J_{F_t}=\beta|x-y|^{-n(1+\beta)},\qquad
 |DF_t|^n=\beta^n|x-y|^{-n(1+\beta)}.
\]
The outer distortion is exactly $\beta^{n-1}\le K$, and $|F_t(x)|=|x-y|^{-\beta}$. Since $p\beta>Q$, we may apply Lemma~\ref{lem:poleintegral} to normalize these mappings as before. The estimates \eqref{eq:peakprofile}--\eqref{eq:uniformconcentration}, with $\sigma$ replaced by $\beta$, use only this modulus. They therefore hold for the normalized radial mappings. The point-evaluation assertion follows by combining the lower bound in \eqref{eq:peakbounds} at $x=a$ with \eqref{eq:evaluationgrowth}.
\end{proof}

\section{Reverse inequalities}\label{sec:dominating}

We now prove the domination theorem and its measure version. The simplest
local case is $d\nu=\mathbf1_F\,dx$, where $F$ occupies a fixed positive
part of a smaller concentric ball. We ask when the integral on $F$
controls the integral on that ball. In a contradiction argument, a
normalized sequence with bounded growth between the two balls has a
nonzero locally uniform limit. Its zeros on the smaller closed ball are
finite, and uniform control of small-ball masses keeps some mass away
from these zeros. The following lemma formulates this observation for
measures, including singular ones.

\begin{lemma}\label{lem:localsmallness}
Fix $0<r<R<1$, $K\ge1$, $0<p<\infty$, and $D,m,A,s>0$. Let $\nu$ be a positive Borel measure on $B(0,r)$, extended by zero to $\R^n$, such that
\[
 \nu(B(0,r))\ge m,\qquad
 \nu(B(z,t))\le At^s
 \quad(z\in\R^n,\ 0<t\le1).
\]
If $g$ is $K$-quasiregular on $B(0,R)$ and
\begin{equation}\label{eq:localdoubling}
 \int_{B(0,R)}|g|^p\dd x
 \le D\int_{B(0,r)}|g|^p\dd x<\infty,
\end{equation}
then
\begin{equation}\label{eq:measurelocalsmallness}
 \int_{B(0,r)}|g|^p\dd x
 \le C\int_{B(0,r)}|g|^p\dd\nu,
\end{equation}
where $C$ depends only on the fixed parameters.
\end{lemma}

\begin{proof}
We may suppose that the integral over the smaller ball is nonzero and normalize it to one. If no uniform constant exists, there are mappings $g_j$ and measures $\nu_j$ satisfying the hypotheses such that
\[
 \int_{B(0,r)}|g_j|^p\dd x=1,\qquad
 \int_{B(0,R)}|g_j|^p\dd x\le D,\qquad
 \int_{B(0,r)}|g_j|^p\dd\nu_j\longrightarrow0.
\]
By the local estimate and quasiregular compactness, we may pass to a diagonal subsequence converging locally uniformly on $B(0,R)$ to a quasiregular mapping or a constant $g$. Since the convergence is uniform on $\overline{B(0,r)}$, we have $\int_{B(0,r)}|g|^p\dd x=1$. Thus $g$ is not identically zero.

The zeros of $g$ in $\overline{B(0,r)}$ are finite. Indeed, a nonzero constant has none, whereas a nonconstant quasiregular mapping has discrete zeros in the larger ball. Denote these zeros by $z_1,\ldots,z_N$. If $N>0$, choose $\varepsilon\in(0,1)$ so that $NA\varepsilon^s<m/2$. The union of the balls $B(z_k,\varepsilon)$ then has $\nu_j$-measure less than $m/2$ for every $j$. On its compact complement in $\overline{B(0,r)}$, the function $|g|$ has a positive minimum. By local uniform convergence, $|g_j|$ is also bounded away from zero there, while the complement has $\nu_j$-measure at least $m/2$. This contradicts the last limit. If $N=0$, the same argument uses the whole closed ball.
\end{proof}

For analytic Bergman spaces, reverse inequalities for measures are part of Luecking's sampling theory~\cite{Lue85,Lue00}. The following condition gives such an inequality for the full class of mappings with fixed distortion. It includes both restricted volume measures and singular measures.

\begin{theorem}\label{thm:singularreverse}
Let $\mu$ be a positive Borel measure on $B$. Suppose that there are $r\in(0,1)$ and $m,M,s>0$ such that, for every $a\in B$, $z\in\R^n$, and $0<t\le1$,
\begin{equation}\label{eq:diffusemeasure}
 \mu(\Delta_r(a))\ge m\delta(a)^Q,
 \qquad
 \mu\bigl(\Delta_r(a)\cap B(z,t\delta(a))\bigr)
 \le Mt^s\delta(a)^Q.
\end{equation}
Then, for every fixed $K\ge1$ and $0<p<\infty$,
\[
 C^{-1}\|f\|_{p,\alpha}^p
 \le\int_B|f|^p\dd\mu
 \le C\|f\|_{p,\alpha}^p,
 \qquad f\in\A^p_{K,\alpha}.
\]
The constant depends only on $n,\alpha,K,p,r,m,M,s$.
\end{theorem}

\begin{proof}
First, $\Delta_r(a)$ is covered by a bounded number of Euclidean balls of radius $\delta(a)$, since its diameter is at most $C_r\delta(a)$. Applying the second condition in \eqref{eq:diffusemeasure} with $t=1$, we obtain
\[
 \mu(\Delta_r(a))\le C_rM\delta(a)^Q.
\]
Pull back the restricted measure to $B(0,r)$ and normalize it by setting
\[
 \nu_a(F)=\delta(a)^{-Q}\mu(\phi_a(F)),
 \qquad F\subset B(0,r)\text{ Borel}.
\]
Its mass is at least $m$. The automorphism $\phi_a$ is Lipschitz on $\overline{B(0,r)}$ with constant $L_r\delta(a)$, by \eqref{eq:automorphismjac} and conformality. Thus the image of $B(z,t)\cap B(0,r)$, when nonempty, lies in a ball of radius $2L_rt\delta(a)$. When $2L_rt\le1$, the second condition in \eqref{eq:diffusemeasure} applies to that ball. Otherwise, the preceding total-mass bound applies. Consequently
\[
 \nu_a(B(z,t))\le At^s\qquad(0<t\le1),
\]
with $A$ independent of $a$.

Choose $r<R<1$ and a covering
\[
 B_j=\Delta_r(a_j),\qquad B_j^*=\Delta_R(a_j)
\]
whose enlarged balls have overlap at most $N$, as in Lemma~\ref{lem:hyperbolic}. Fix $f\in\A^p_{K,\alpha}$, and write
\[
 A_j=\int_{B_j}|f|^p\,dV_\alpha,\qquad
 A_j^*=\int_{B_j^*}|f|^p\,dV_\alpha.
\]
Call $j$ good if $A_j^*\le2NA_j$. Since $f\in\A^p_{K,\alpha}$, it follows from bounded overlap that $\sum_jA_j^*<\infty$. The remaining balls account for at most half the integral:
\[
 \sum_{j\text{ not good}}A_j
 \le\frac1{2N}\sum_jA_j^*
 \le\frac12\|f\|_{p,\alpha}^p.
\]
Since the smaller balls cover $B$, we obtain
\begin{equation}\label{eq:goodmass}
 \sum_{j\text{ good}}A_j\ge\frac12\|f\|_{p,\alpha}^p.
\end{equation}
For a good index, apply Lemma~\ref{lem:localsmallness} to $g=f\circ\phi_{a_j}$ and $\nu=\nu_{a_j}$. By the weighted change of variables \eqref{eq:weighttransport}, we have
\[
 A_j\asymp\delta(a_j)^Q\int_{B(0,r)}|g|^p\dd u,\qquad
 A_j^*\asymp\delta(a_j)^Q\int_{B(0,R)}|g|^p\dd u.
\]
Hence \eqref{eq:localdoubling} holds with $D$ independent of $j$, and \eqref{eq:measurelocalsmallness} becomes
\[
 A_j\le C\int_{B_j}|f|^p\dd\mu.
\]
Summing over good indices and using \eqref{eq:goodmass} together with bounded overlap, we obtain the reverse inequality.

For the forward inequality, we combine \eqref{eq:hyperbolicsup} with the upper mass bound to obtain
\[
 \int_{B_j}|f|^p\dd\mu
 \le\mu(B_j)\sup_{B_j}|f|^p
 \le CA_j^*.
\]
Summation proves the result. In particular, applying this estimate to a nonzero constant shows that $\mu(B)<\infty$.
\end{proof}

\begin{proof}[Proof of Theorem~\ref{thm:maindomination}]
Suppose that $E$ satisfies \eqref{eq:density}, and set $\mu=\mathbf1_E\,dV_\alpha$. The lower bound in \eqref{eq:diffusemeasure} follows from the weighted volume estimate. On $\Delta_r(a)$ we have $\delta(x)\asymp\delta(a)$, so
\[
 \mu\bigl(\Delta_r(a)\cap B(z,t\delta(a))\bigr)
 \le C\delta(a)^\alpha(t\delta(a))^n
 =Ct^n\delta(a)^Q.
\]
Sufficiency now follows from Theorem~\ref{thm:singularreverse} with $s=n$.

For necessity, assume $K\ge\Kstar$ and test \eqref{eq:domination} on the norm-one mappings $h_a$ from Section~\ref{sec:concentration}. We have
\[
 \int_E|h_a|^p\,dV_\alpha\ge C_E^{-1}.
\]
By Lemma~\ref{lem:concentration}, choose one $r<1$, independent of $a$, for which the integral over $B\setminus\Delta_r(a)$ is at most $(2C_E)^{-1}$. Hence
\[
 \int_{E\cap\Delta_r(a)}|h_a|^p\,dV_\alpha\ge(2C_E)^{-1}.
\]
The upper bound in \eqref{eq:peakbounds} yields
$V_\alpha(E\cap\Delta_r(a))\ge c\delta(a)^Q$.
By \eqref{eq:hyperbolicvolume}, this is \eqref{eq:density}. Under \eqref{eq:radialrange}, the same argument applies with the normalized radial mappings from Proposition~\ref{prop:radialrange}.
\end{proof}

\begin{corollary}\label{cor:boundeddensity}
Let $0\le w\le M<\infty$ be measurable on $B$. If, for some $r\in(0,1)$ and $\eta>0$,
\begin{equation}\label{eq:lowerdensity}
 \int_{\Delta_r(a)}w\,dV_\alpha
 \ge\eta V_\alpha(\Delta_r(a))\qquad(a\in B),
\end{equation}
then, for every fixed $K$,
\begin{equation}\label{eq:reverseweight}
 \|f\|_{p,\alpha}^p\le C\int_Bw|f|^p\,dV_\alpha
 \qquad(f\in\A^p_{K,\alpha}).
\end{equation}
Conversely, \eqref{eq:reverseweight} implies \eqref{eq:lowerdensity} for some fixed $r,\eta$ if $K\ge\Kstar$ or \eqref{eq:radialrange} holds.
\end{corollary}

\begin{proof}
The measure $w\,dV_\alpha$ satisfies \eqref{eq:diffusemeasure} with $s=n$: the lower bound follows from \eqref{eq:lowerdensity}, and the upper bound follows from $w\le M$ and the volume estimate used above. We may therefore apply Theorem~\ref{thm:singularreverse} to obtain \eqref{eq:reverseweight}.

Conversely, test \eqref{eq:reverseweight} on $h_a$. In view of the bound $w\le M$, we may use uniform concentration to choose one $r<1$ for which
\[
 \int_{\Delta_r(a)}w|h_a|^p\,dV_\alpha\ge(2C)^{-1}.
\]
Since $|h_a|^p\le C_1\delta(a)^{-Q}$, this implies
$\int_{\Delta_r(a)}w\,dV_\alpha\ge(2CC_1)^{-1}\delta(a)^Q$.
It now follows from the weighted volume estimate that \eqref{eq:lowerdensity} holds. The same proof applies in the radial range.
\end{proof}

\subsection{Singular and atomic measures}

Theorem~\ref{thm:singularreverse} permits measures supported on sets of volume zero. For example, let $(a_j)$ be a maximal $h$-separated sequence in the hyperbolic metric $d$, where $h>0$ is fixed. Set $\delta_j=\delta(a_j)$ and choose $0<\varepsilon<1/4$. The measure
\[
 \mu=\sum_j\delta_j^{\alpha+1}
       \mathcal H^{n-1}\!\restriction S_j,
 \qquad S_j=\partial B(a_j,\varepsilon\delta_j),
\]
is carried by a countable union of spheres and is therefore singular with respect to $dV_\alpha$. Here $\mathcal H^{n-1}\!\restriction S_j$ denotes surface measure on $S_j$.

To check \eqref{eq:diffusemeasure}, first note that every point of $S_j$ has hyperbolic distance at most
$c_\varepsilon=2\varepsilon/(1-\varepsilon)$ from $a_j$. Indeed, integrate the length element $2|dx|/(1-|x|^2)$ along the straight segment, where $\delta(x)\ge(1-\varepsilon)\delta_j$. Choose $r<1$ with $2\arctanh r>h+c_\varepsilon$. By maximality, for every $a\in B$ there is a point $a_j$ with $d(a,a_j)\le h$. Hence $S_j\subset\Delta_r(a)$ and $\delta_j\asymp\delta(a)$, so
\[
 \mu(\Delta_r(a))\ge c_n\varepsilon^{n-1}\delta_j^Q
 \ge c\delta(a)^Q.
\]
If $S_j$ meets $\Delta_r(a)$, its center lies within hyperbolic distance $2\arctanh r+c_\varepsilon$ of $a$. By separation and invariant-volume packing, the number of these indices is uniformly bounded, and all corresponding $\delta_j$ are comparable to $\delta(a)$. The elementary sphere estimate
\[
 \mathcal H^{n-1}(S_j\cap B(z,u))\le C_nu^{n-1}
\]
is independent of the radius of $S_j$. To see this, split the sphere into the $2n$ regions where a specified signed coordinate of the unit normal is at least $1/\sqrt n$, and project each region onto the corresponding coordinate hyperplane. The projection has Jacobian at least $1/\sqrt n$, and its image lies in a ball of radius $u$. Summing the estimates gives
\[
 \mu\bigl(\Delta_r(a)\cap B(z,t\delta(a))\bigr)
 \le C\delta(a)^{\alpha+1}(t\delta(a))^{n-1}
 =Ct^{n-1}\delta(a)^Q.
\]
Thus \eqref{eq:diffusemeasure} holds with $s=n-1$. The measure is finite as well: the disjoint hyperbolic balls of radius $h/2$ about the $a_j$ have weighted volumes comparable to $\delta_j^Q$, so $\sum_j\delta_j^Q<\infty$.

Atomic measures can also control the full fixed-distortion class, although they do not satisfy \eqref{eq:diffusemeasure}. This follows from the local H\"older estimate, provided the points are sufficiently close together.

Choose a Whitney decomposition $\mathcal W$ of $B$ into cubes with disjoint interiors such that
\begin{equation}\label{eq:whitneygeometry}
 4W\Subset B,\qquad
 \delta(x)\asymp\ell(W)\quad(x\in4W),
\end{equation}
and the cubes $4W$ have bounded overlap. Here $\ell(W)$ is the side length. Such a decomposition is obtained from maximal dyadic cubes whose distance from $\partial B$ is at least a sufficiently large fixed multiple of their diameter. The reverse distance bound follows because the dyadic parent fails this condition. If two enlarged cubes meet, their side lengths are comparable by these bounds and the Lipschitz property of $\delta$. Bounded overlap then follows from disjointness of the original interiors.

\begin{proposition}\label{thm:sampling}
There is an integer $m=m(n,K,p,\alpha)$ with the following property. Subdivide each $W\in\mathcal W$ into $m^n$ congruent cubes $E$, and choose $a_E\in E$. The atomic measure of finite mass
\[
 \mu_m=\sum_E V_\alpha(E)\,\delta_{a_E}
\]
satisfies
\[
 \|f\|_{p,\alpha}^p
 \asymp\int_B|f|^p\dd\mu_m
 =\sum_E V_\alpha(E)|f(a_E)|^p
 \qquad(f\in\A^p_{K,\alpha}).
\]
Here $\delta_a$ denotes unit point mass at $a$. The constants are independent of the choices of the points.
\end{proposition}

\begin{proof}
The measure has total mass $\mu_m(B)=V_\alpha(B)<\infty$. Let $S_mf$ equal $f(a_E)$ on $E$, assigning cube boundaries arbitrarily. Take $m\ge4\sqrt n$. If $x,a_E\in E\subset W$, then $B(a_E,\ell(W))\subset4W$ and $|x-a_E|\le\ell(W)/4$. We apply the local H\"older estimate \eqref{eq:localholder} with radius $\ell(W)/2$ and then use \eqref{eq:whitneygeometry} to obtain
\[
 \sup_{x\in E}|f(x)-f(a_E)|^p
 \le Cm^{-p\vartheta}\ell(W)^{-Q}
                  \int_{4W}|f|^p\,dV_\alpha.
\]
Multiply by $V_\alpha(E)$ and sum. Since
$\sum_{E\subset W}V_\alpha(E)=V_\alpha(W)\asymp\ell(W)^Q$,
it follows from bounded overlap that
\begin{equation}\label{eq:samplingerror}
 \|f-S_mf\|_{p,\alpha}^p
 \le Cm^{-p\vartheta}\|f\|_{p,\alpha}^p.
\end{equation}
For $p\ge1$, choose $m$ so that the $p$th root of the coefficient is less than $1/2$ and use the triangle inequality in both directions. For $0<p<1$, use
\[
 \|f\|_{p,\alpha}^p
 \le\|f-S_mf\|_{p,\alpha}^p+\|S_mf\|_{p,\alpha}^p
\]
and the corresponding inequality with $f$ and $S_mf$ interchanged. Both bounds again follow if we take the coefficient in \eqref{eq:samplingerror} less than $1/2$. Finally,
$\|S_mf\|_{p,\alpha}^p=\int_B|f|^p\dd\mu_m$.
\end{proof}

The sampling points just constructed form a locally finite set. The values at these points control the integral, but they need not determine the mapping. The following perturbations show this for every $K>1$.

\begin{proposition}\label{prop:nonlinear}
Let $K>1$. For every locally finite subset $S\subset B$, there are distinct bounded mappings $f,g\in\A^p_{K,\alpha}$ which agree on $S$. Moreover, no constant $C$ satisfies
\[
 |f(0)-g(0)|\le C\|f-g\|_{p,\alpha}
 \qquad(f,g\in\A^p_{K,\alpha}).
\]
\end{proposition}

\begin{proof}
Choose $0<c<1$ so that $((1+c)/(1-c))^n\le K$. If a smooth compactly supported perturbation $\psi:\R^n\to\R^n$ satisfies $\|D\psi\|_\infty\le c$, the singular values of $I+D\psi$ lie between $1-c$ and $1+c$. Its determinant is positive, since $I+tD\psi$ is nonsingular for $0\le t\le1$. Thus $x\mapsto x+\psi(x)$ has outer distortion at most $K$. It is also injective, since
\[
 |x+\psi(x)-y-\psi(y)|\ge(1-c)|x-y|.
\]
Its restriction to $B$ is therefore quasiconformal onto its image.

Since $S$ is locally finite, we can choose a ball $U\Subset B$ disjoint from $S$. Choose a nonzero smooth perturbation supported in $U$ and multiply it by a sufficiently small constant to obtain the preceding derivative bound. The identity and its perturbation are distinct bounded members of the class and agree on $S$.

For the second assertion, choose $\eta\in C_c^\infty(B)$ with $\eta(0)=1$ and $b>0$ such that $b\|\nabla\eta\|_\infty\le c$. Set
\[
 f(x)=x,\qquad
 g_\varepsilon(x)=x+b\varepsilon\eta(x/\varepsilon)e_1,
 \qquad 0<\varepsilon<1/2.
\]
These maps have the same distortion bound, while
\[
 |g_\varepsilon(0)-f(0)|=b\varepsilon,\qquad
 \|g_\varepsilon-f\|_{p,\alpha}\asymp
 b\varepsilon^{1+n/p}.
\]
The second estimate follows by changing variables on the support $B(0,\varepsilon)$, where the weight is uniformly comparable to one. The ratio tends to infinity as $\varepsilon\downarrow0$.
\end{proof}

\section{The conformal class and the critical exponent}\label{sec:distortion}

We now determine what changes when the distortion bound is one. Throughout
this section, $n\ge3$ and $Q=n+\alpha$. By the quasiregular Liouville theorem~\cite{Rickman},
every nonconstant $1$-quasiregular mapping on $B$ is a restriction of a
M\"obius transformation. It is either a similarity or has the form
\[
 f(x)=b+\lambda U\frac{x-y}{|x-y|^2},\qquad
 \lambda>0,\quad |y|\ge1,
\]
where $U$ is an orthogonal map with determinant $-1$. The pole cannot belong
to $B$. This representation makes it possible to describe the conformal
Bergman classes more precisely than a general distortion class.

\begin{lemma}\label{lem:conformal-pole-size}
For $0<d\le1$ and $\omega\in\partial B$, put
\[
 I_p(d)=\int_B|x-(1+d)\omega|^{-p}\,dV_\alpha(x).
\]
The integral is independent of $\omega$, and
\[
 I_p(d)\asymp
 \begin{cases}
  1,&p<Q,\\
  \log(e/d),&p=Q,\\
  d^{Q-p},&p>Q.
 \end{cases}
\]
A pole on $\partial B$ is integrable precisely when $p<Q$.
\end{lemma}

\begin{proof}
The first assertion follows from rotation invariance. The supercritical estimate
and the boundary-pole assertion were proved in Lemma~\ref{lem:poleintegral}.
When $p<Q$, write $y=(1+d)\omega$. Since
\[
 |x-y|^2-|x-\omega|^2
   =d(2+d-2\langle x,\omega\rangle)\ge0,
\]
the boundary-pole integral is a uniform upper bound. By integrating over
$B(0,1/4)$, we also obtain a positive uniform lower bound.

It remains to consider $p=Q$. Since $|x-y|\ge d$, we use the shells
$2^jd\le |x-y|<2^{j+1}d$. By the weighted cap estimate, the integral
over each shell of radius at most four is bounded above by $C$. There are
at most $C\log(e/d)$ such shells, and they cover $B$. Shells whose radius is at least $1/2$ contribute only a bounded amount, by $V_\alpha(B)<\infty$; thus the cap estimate is used only at small radii. For the lower bound,
use the pairwise disjoint balls
\[
 B((1-t_j)\omega,t_j/8),\qquad t_j=4^{-j},\quad d\le t_j\le1/4.
\]
On each ball, $\delta(x)\asymp t_j$ and $|x-y|\asymp t_j$, so its
contribution to $I_Q(d)$ is bounded below by a positive constant. The
number of these balls is comparable to $\log(1/d)$ for sufficiently
small $d$. For the remaining values of $d$, integration over a fixed
interior ball supplies the lower bound. This proves the lemma.
\end{proof}

The translation term can partly cancel the pole term. We therefore first control the parameters in the M\"obius representation by the values of $f$ on an interior ball.
Set $S_f=\sup_{B(0,1/4)}|f|$. By the local estimate,
\[
 S_f\le C\|f\|_{p,\alpha}.
\]
For a nonsimilarity, we use the inversion identity to write
\[
 |f(x)-f(z)|
   =\frac{\lambda|x-z|}{|x-y|\,|z-y|}.
\]
Taking $z=0$ and $x=e_1/8$, and using
$|e_1/8-y|\le(9/8)|y|$, we obtain
\[
 \frac{\lambda}{|y|^2}
 =8\frac{|e_1/8-y|}{|y|}\,|f(e_1/8)-f(0)|
 \le CS_f.
\]
Applying the same distance identity with $z=0$ and arbitrary $x\in B$
therefore yields
\begin{equation}\label{eq:conformal-parameter-growth}
 |f(x)|\le CS_f\left(1+\frac{|y|}{|x-y|}\right).
\end{equation}
If $|y|\ge2$, the ratio on the right is at most two, so
$\sup_B|f|\le CS_f$. If $1\le|y|\le2$, the scale parameter satisfies
$\lambda\le CS_f$, and by evaluating at zero we find
\[
 |b|\le |f(0)|+\frac{\lambda}{|y|}\le CS_f.
\]
Constants and similarities also satisfy $\sup_B|f|\le CS_f$.
All these constants are independent of the M\"obius parameters.

\subsection{Point evaluation}

We can now determine the sharp order of point evaluation in each range from the pole integral.

\begin{proposition}\label{prop:conformalevaluation}
For $n\ge3$, $\alpha>-1$, and $0<p<\infty$,
\begin{equation}\label{eq:conformalevaluation}
 \mathcal E_{1,p,\alpha}(a)\asymp
 \begin{cases}
  \delta(a)^{-1},&p<Q,\\
  \delta(a)^{-1}\bigl[\log(e/\delta(a))\bigr]^{-1/Q},&p=Q,\\
  \delta(a)^{-Q/p},&p>Q.
 \end{cases}
\end{equation}
The constants depend only on $n,p,\alpha$.
\end{proposition}

\begin{proof}
Write $t=\delta(a)$ and normalize $\|f\|_{p,\alpha}=1$. Constants,
similarities, and poles with $|y|\ge2$ are uniformly bounded by the
preceding estimates. Suppose $1\le|y|<2$ and put $d=|y|-1$. Using the bounds
on $b$ and applying the elementary power inequality to
$\lambda U(x-y)/|x-y|^2=f(x)-b$, we find that
\[
 \lambda I_p(d)^{1/p}\le C.
\]
When $d=0$, this case is possible only for $p<Q$. Since
$|a-y|\ge t+d$, we obtain, for $d>0$,
\[
 |f(a)|\le C+\frac{C}{(t+d)I_p(d)^{1/p}}.
\]
For $p<Q$, the same estimate with $I_p(0)<\infty$ covers boundary
poles and gives the upper bound $Ct^{-1}$.

When $p=Q$, write $L(s)=\log(e/s)$. If $d\le t$, then
$(t+d)L(d)^{1/Q}\ge tL(t)^{1/Q}$. If $d\ge t$, use that
$sL(s)^{1/Q}$ is increasing on $(0,1]$: its derivative is
$L(s)^{1/Q}(1-1/(QL(s)))>0$, because $Q>1$. This proves the critical
upper bound. When $p>Q$, put $\gamma=1-Q/p\in(0,1)$. The pole estimate
reduces the upper bound to
\[
 \frac{d^\gamma}{t+d}\le t^{\gamma-1},
\]
which follows by considering $d\le t$ and $d\ge t$ separately.

For the matching lower bounds, take $a=(1-t)\omega$, with an arbitrary
$\omega$ if $a=0$, and choose a fixed orientation-reversing orthogonal
map $U$. For $p<Q$, normalize the boundary-pole mapping
$U(x-\omega)/|x-\omega|^2$ to have $\|\cdot\|_{p,\alpha}=1$. For $p\ge Q$,
use the exterior-pole mapping
$U(x-(1+t)\omega)/|x-(1+t)\omega|^2$, again normalized in the Bergman
class. Its value at $a$ has modulus $1/(2t)$ before normalization.
The stated lower bounds now follow from Lemma~\ref{lem:conformal-pole-size}.
\end{proof}

\subsection{Compactness and domination below the threshold}

We prove the compactness and subcritical domination assertions of Theorem~\ref{thm:mainconformal}. The domination conclusion is uniform over measurable sets $E$ whose weighted measure is bounded below: for each $\eta>0$, its constant depends only on $n,p,\alpha,\eta$ when $V_\alpha(E)\ge\eta$.

\begin{proof}
We first prove compactness below the threshold. On the closed unit ball,
$S_f$ is uniformly bounded. Formula~\eqref{eq:conformal-parameter-growth}
shows that constants, similarities, and poles with $|y|\ge2$ are uniformly
bounded on $B$. For the remaining mappings, put $\omega=y/|y|$ and use
$|x-y|\ge|x-\omega|$ to obtain
\[
 |f(x)|^p\le C(1+|x-\omega|^{-p}).
\]
For $0<r<1/4$, we apply the weighted cap estimate on inward dyadic shells to obtain
\[
 \int_{B\cap B(\omega,r)}|x-\omega|^{-p}\,dV_\alpha
 \le C\sum_{j\ge0}(2^{-j}r)^{Q-p}\le Cr^{Q-p},
\]
uniformly in $\omega$. On the complement of this cap, the integrand is
at most $r^{-p}$. Hence, for every $\varepsilon>0$,
\[
 \sup_{\|f\|_{p,\alpha}\le1}
 \int_{\{\delta<\varepsilon\}}|f|^p\,dV_\alpha
 \le C r^{Q-p}+C(1+r^{-p})V_\alpha(\{\delta<\varepsilon\}).
\]
First make $r$ small, using $p<Q$, and then let $\varepsilon$ tend to
zero. The right-hand side tends to zero uniformly in the unit ball.
By quasiregular compactness, every sequence in that ball has a locally
uniformly convergent subsequence. Its limit remains in the same ball by
Fatou's lemma. With the uniform tail estimate, we can then conclude that
the convergence holds in the metric $d_p$. Thus the closed unit ball is
compact.

To prove the assertion about dominating sets, suppose that the uniform
inequality for some fixed $\eta>0$ were false. There would be mappings
$f_j$ and measurable sets $E_j$ such that
\[
 \|f_j\|_{p,\alpha}=1,\qquad V_\alpha(E_j)\ge\eta,\qquad
 \int_{E_j}|f_j|^p\,dV_\alpha\longrightarrow0.
\]
By compactness, we may pass to a subsequence converging in the metric $d_p$ to
$f$ with $\|f\|_{p,\alpha}=1$. The elementary power inequality then
implies $\int_{E_j}|f|^p\,dV_\alpha\to0$. But $f$ is not identically
zero, and its zero set has weighted measure zero: it is empty for a
nonzero constant and discrete otherwise. Therefore, for sufficiently
small $c>0$,
\[
 V_\alpha(\{|f|\le c\})<\eta/2.
\]
It follows that $\int_{E_j}|f|^p\,dV_\alpha\ge c^p\eta/2$, a
contradiction. The converse follows by testing on a nonzero constant.

Finally, let $p\ge Q$ and let $d_j\downarrow0$. The simple-pole maps
with poles $(1+d_j)e_1$, divided by $I_p(d_j)^{1/p}$, have norm one
and tend locally uniformly to zero, because $I_p(d_j)\to\infty$.
They cannot have a subsequence converging in the metric $d_p$:
such a limit would be zero by local uniform convergence, while
their norms remain one. This proves noncompactness at and above
the threshold.
\end{proof}

The pole calculation also shows that, for fixed $n\ge3$ and $\alpha$,
the classes $\A^p_{1,\alpha}$ are the same for all $p<Q$. For $p\ge Q$
they consist precisely of the bounded conformal mappings and constants.
Indeed, a M\"obius restriction is bounded unless its pole lies on
$\partial B$, and a boundary pole is integrable exactly when $p<Q$.

\subsection{Domination at the critical exponent}

We now prove the logarithmic-density criterion in
Theorem~\ref{thm:mainconformal}. The essential point is that the critical
pole integral is comparable, up to a bounded term, to the integral of
the cap densities with respect to $dr/r$.

\begin{proof}[Proof of the critical criterion in Theorem~\ref{thm:mainconformal}]
Fix $r_0\in(0,1)$. We first compare the density integral with a pole
integral. Set
\[
 A_E(\omega,r)=V_\alpha(E\cap B(\omega,r)),\qquad
 J_E(\omega,d)=\int_d^{r_0}\Theta_E(\omega,r)\,\frac{dr}{r},
\]
and, for $0<d<r_0$, put
\[
 I_E(\omega,d)=\int_E|x-(1+d)\omega|^{-Q}\,dV_\alpha(x).
\]
We claim that, uniformly in $E$, $\omega$, and $d$,
\begin{equation}\label{eq:critical-cap-pole}
 cJ_E(\omega,d)\le I_E(\omega,d)
 \le C\bigl(1+J_E(\omega,d)\bigr).
\end{equation}
The constants may depend on $n$, $\alpha$, and $r_0$.

To prove this, write $s=|x-\omega|$. The identity
\[
 |x-(1+d)\omega|^2
 =s^2+2d(1-\langle x,\omega\rangle)+d^2
\]
together with the triangle inequality implies
\[
 \frac{d+s}{\sqrt2}\le |x-(1+d)\omega|\le d+s.
\]
For $s<r_0$, the quantities $d+s$ and $\max(d,s)$ differ by at
most a factor of two. Moreover, by Tonelli's theorem,
\[
 \begin{split}
 &\int_{E\cap B(\omega,r_0)}
          \max(d,|x-\omega|)^{-Q}\,dV_\alpha(x)\\
 &\qquad=r_0^{-Q}A_E(\omega,r_0)
       +Q\int_d^{r_0}A_E(\omega,r)\,\frac{dr}{r^{Q+1}}.
 \end{split}
\]
Indeed, the integrand on the left is
$r_0^{-Q}+Q\int_{\max(d,s)}^{r_0}r^{-Q-1}\,dr$.
By the weighted cap estimate in Lemma~\ref{lem:cap}, the last
integral is comparable to $J_E(\omega,d)$. The first term is bounded
uniformly. Outside $B(\omega,r_0)$, the pole integral is at most
$r_0^{-Q}V_\alpha(B)$. These observations prove
\eqref{eq:critical-cap-pole}.

Suppose first that $E$ is dominating. Apply the domination inequality
to the conformal pole mapping
\[
 x\longmapsto U\frac{x-(1+d)\omega}{|x-(1+d)\omega|^2},
\]
where $U$ is any fixed orientation-reversing orthogonal map. It
belongs to $\mathcal A^Q_{1,\alpha}$ for every $d>0$. Thus
\[
 I_E(\omega,d)\ge c I_Q(d)\ge c'\log(e/d),
\]
by Lemma~\ref{lem:conformal-pole-size}. Combining this with the upper
bound in \eqref{eq:critical-cap-pole}, we obtain
$J_E(\omega,d)\ge c''\log(e/d)$ for all sufficiently small $d$,
uniformly in $\omega$. This proves \eqref{eq:logdensity}.

Conversely, assume \eqref{eq:logdensity}. The lower bound
in \eqref{eq:critical-cap-pole} and
Lemma~\ref{lem:conformal-pole-size} imply that
\begin{equation}\label{eq:critical-pole-domination}
 I_E(\omega,d)\ge c_E I_Q(d),\qquad
 \omega\in\partial B,\quad 0<d\le1.
\end{equation}
Initially this follows for all sufficiently small $d$. The condition
also forces $V_\alpha(E)>0$. For the remaining $d\le1$, use
$|x-(1+d)\omega|\le3$ on $B$, so that
$I_E(\omega,d)\ge3^{-Q}V_\alpha(E)$, while $I_Q(d)$ is uniformly
bounded when $d$ is bounded away from zero.

A general M\"obius mapping also contains a translation term. We first
control this term from the values on a fixed interior part of $E$.
Choose $R<1$ so that $F=E\cap B(0,R)$ has positive weighted measure.
We claim that
\begin{equation}\label{eq:conformal-interior-observation}
 S_f\le C_F\left(\int_F|f|^Q\,dV_\alpha\right)^{1/Q}
\end{equation}
for every conformal mapping on $B$, with constants included. To see
this, suppose otherwise and normalize a sequence so that $S_{f_j}=1$
and $\int_F|f_j|^Q\,dV_\alpha\to0$. It follows from the parameter estimate
\eqref{eq:conformal-parameter-growth} and its elementary
analogue for similarities that this sequence is bounded on every compact
subset of $B$. By quasiregular compactness, we may pass to a locally uniformly
convergent subsequence with limit $f$, where $S_f=1$. The limit is
conformal or constant. Its zero set has weighted measure zero, since
it is discrete or empty. On the other hand, by uniform convergence on
$\overline{B(0,R)}$, we have $\int_F|f|^Q\,dV_\alpha=0$, a
contradiction. This proves \eqref{eq:conformal-interior-observation}.

It follows that $S_f\le C_E\|f\|_{L^Q(E,dV_\alpha)}$. Constants,
similarities, and mappings with poles $|y|\ge2$ are now controlled
on all of $B$ by \eqref{eq:conformal-parameter-growth} and the
preceding parameter bounds. A boundary pole is not integrable at $p=Q$,
so the remaining mappings have the form
\[
 f(x)=b+\lambda U\frac{x-y}{|x-y|^2},\qquad
 y=(1+d)\omega,\quad 0<d<1.
\]
By the parameter bound, $|b|\le CS_f$. Since $Q>1$, we may use the
triangle inequality on $E$ to obtain
\[
 \lambda I_E(\omega,d)^{1/Q}
 \le \|f\|_{L^Q(E,dV_\alpha)}+|b|V_\alpha(E)^{1/Q}
 \le C_E\|f\|_{L^Q(E,dV_\alpha)}.
\]
Applying \eqref{eq:critical-pole-domination} and then the triangle
inequality on $B$, we conclude that
\[
 \|f\|_{Q,\alpha}
 \le |b|V_\alpha(B)^{1/Q}+\lambda I_Q(d)^{1/Q}
 \le C_E\|f\|_{L^Q(E,dV_\alpha)}.
\]
This is the required domination inequality.

Finally, $0\le\Theta_E\le1$. Replacing $r_0$ by another fixed
radius changes the numerator in \eqref{eq:logdensity}
by a bounded amount, uniformly in $\omega$ and small $d$, and
changes its denominator by a constant. The normalized averages
therefore differ by a quantity tending uniformly to zero. This
proves the asserted independence of $r_0$.
\end{proof}

The criterion permits gaps whose logarithmic widths tend to infinity,
provided they occupy a sufficiently small proportion of the logarithmic scales. This gives the
following explicit distinction from relative density.

\begin{proposition}\label{prop:conformal-critical-domination}
Suppose $n\ge3$ and $p=Q$. The set
\[
 E=\left\{x\in B:
 -\log\delta(x)\notin\bigcup_{j=1}^{\infty}(4^j,4^j+j)\right\}
\]
is dominating for $\mathcal A^Q_{1,\alpha}$ but is not relatively
dense.
\end{proposition}

\begin{proof}
Let $H=\bigcup_{j\ge1}(4^j,4^j+j)$, and fix $r_0<1/4$.
For $r=e^{-u}<r_0$, the ball
\[
 D_{\omega,r}=B((1-r/2)\omega,r/8)
\]
lies in $B\cap B(\omega,r)$, has weighted measure comparable to
$r^Q$, and satisfies
\[
 3r/8\le\delta(x)\le5r/8,\qquad x\in D_{\omega,r}.
\]
Consequently $D_{\omega,r}\subset E$ whenever
\[
 [u+\log(8/5),u+\log(8/3)]\cap H=\varnothing.
\]
For these values of $u$, it follows from Lemma~\ref{lem:cap} that
$\Theta_E(\omega,e^{-u})\ge c$, uniformly in $\omega$.
The excluded values of $u$ belong to the intervals
\[
 (4^j-\log(8/3),\,4^j+j-\log(8/5)),\qquad j\ge1.
\]
Each has length $j+\log(5/3)$. At most $C\log(U+2)$ such
intervals meet $[\log(1/r_0),U]$, and the sum of their lengths is
at most $C\log^2(U+2)=o(U)$. Hence, as $d=e^{-U}\downarrow0$,
\[
 \begin{split}
 \int_d^{r_0}\Theta_E(\omega,r)\,\frac{dr}{r}
 &=\int_{\log(1/r_0)}^U\Theta_E(\omega,e^{-u})\,du\\
 &\ge c\bigl(U-\log(1/r_0)-C\log^2(U+2)\bigr).
 \end{split}
\]
The lower bound is uniform in $\omega$, so domination follows from
Theorem~\ref{thm:mainconformal}.

To see that relative density fails, put
$t_j=\exp(-4^j-j/2)$ and $a_j=(1-t_j)e_1$. For any fixed
$r<1$, we have, by the hyperbolic-ball estimates,
$c_rt_j\le\delta(x)\le C_rt_j$ on $\Delta_r(a_j)$. Thus
$-\log\delta(x)$ differs from $4^j+j/2$ by a uniformly bounded
amount. For all sufficiently large $j$, the entire ball lies in
the deleted shell $4^j<-\log\delta(x)<4^j+j$. Hence
$E\cap\Delta_r(a_j)=\varnothing$. This holds for every fixed
$r<1$, so $E$ is not relatively dense.
\end{proof}

For $p>Q$, Theorem~\ref{thm:maindomination} applies with $K=1$, since
\eqref{eq:radialrange} holds. This proves the remaining part of
Theorem~\ref{thm:mainconformal}. The example above shows why the
inequality $p>Q$ is strict in the relative-density criterion.

The logarithmic behavior is specific to the conformal class. For every
fixed $K>1$, relative density becomes necessary at the critical exponent,
and the logarithmic factor in point evaluation disappears.

\begin{corollary}\label{cor:criticaldistortion}
Let $n\ge3$, $\alpha>-1$, and $K>1$. A measurable set $E\subset B$ is
dominating for $\A^Q_{K,\alpha}$ if and only if it is relatively dense.
Moreover,
\[
 \mathcal E_{K,Q,\alpha}(a)\asymp\delta(a)^{-1},\qquad a\in B,
\]
with constants depending only on $n,K,\alpha$. In particular, the set in
Proposition~\ref{prop:conformal-critical-domination} dominates
$\A^Q_{1,\alpha}$ but does not dominate $\A^Q_{K,\alpha}$ for any $K>1$.
\end{corollary}

\begin{proof}
At $p=Q$, condition~\eqref{eq:radialrange} holds for every $K>1$.
The assertion about sets therefore follows from Theorem~\ref{thm:maindomination}.
For point evaluation, we use \eqref{eq:evaluationgrowth} for the upper bound.
For the lower bound, put $\beta=K^{1/(n-1)}>1$ and write
$a=(1-t)\omega$ with $t=\delta(a)$, taking any $\omega\in\partial B$
when $a=0$. Let $U$ be an
orientation-reversing orthogonal map and put
\[
 F_t(x)=\frac{U(x-(1+t)\omega)}{|x-(1+t)\omega|^{1+\beta}}.
\]
As computed in Proposition~\ref{prop:radialrange}, this mapping has outer
distortion $\beta^{n-1}\le K$. Since $Q\beta>Q$, we obtain from Lemma~\ref{lem:poleintegral} that
\[
 \|F_t\|_{Q,\alpha}^Q\asymp t^{Q(1-\beta)},
 \qquad |F_t(a)|=(2t)^{-\beta}.
\]
Dividing $F_t$ by $\|F_t\|_{Q,\alpha}$ yields
$\mathcal E_{K,Q,\alpha}(a)\ge ct^{-1}$. The last assertion follows
because the set in Proposition~\ref{prop:conformal-critical-domination}
is not relatively dense.
\end{proof}

\begin{remark*}
The contrast in Corollary~\ref{cor:criticaldistortion} concerns boundary
growth. At interior points, evaluation is right-continuous in the
distortion bound: for fixed $p>0$, $\alpha>-1$, and $K_0\ge1$,
\[
 \mathcal E_{K,p,\alpha}(a)\longrightarrow
 \mathcal E_{K_0,p,\alpha}(a)\qquad(K\downarrow K_0),
\]
locally uniformly for $a\in B$.

To prove this, note that the suprema are monotone in $K$. Fix $a\in B$
and let $K_j\downarrow K_0$. Choose $f_j\in\A^p_{K_j,\alpha}$ with
$\|f_j\|_{p,\alpha}\le1$ and
$|f_j(a)|\ge\mathcal E_{K_j,p,\alpha}(a)-1/j$. All these mappings have
distortion at most $K_1$, so a subsequence converges locally uniformly
to $f$, with $\|f\|_{p,\alpha}\le1$ by Fatou's lemma. For every positive
integer $m$, the tail has distortion at most $K_0+1/m$. By the limit theorem,
the same bound holds for $f$. Letting $m\to\infty$, we see that $f$
has distortion at most $K_0$. We may now pass to the limit at $a$ to obtain
\[
 \limsup_{j\to\infty}\mathcal E_{K_j,p,\alpha}(a)
 \le\mathcal E_{K_0,p,\alpha}(a).
\]
The reverse inequality follows from
$\A^p_{K_0,\alpha}\subset\A^p_{K_j,\alpha}$.
By Lemma~\ref{lem:local}, the unit balls with $K\le K_1$ have a common
local H\"older modulus, as do their evaluation suprema.
A finite covering argument then shows that the convergence is uniform on each compact set.

At $p=Q$, evaluation therefore approaches its conformal value on every
compact subset as $K\downarrow1$. The boundary orders in
\eqref{eq:conformalevaluation} and Corollary~\ref{cor:criticaldistortion}
show that the convergence is not uniform on $B$.
\end{remark*}

\subsection{Compact embeddings beyond the Carleson condition}

The estimates in \eqref{eq:conformalevaluation} yield compact embeddings for which the usual Bergman Carleson condition fails, for every positive target exponent. Here an inclusion is compact if it is continuous in the ambient metrics and maps bounded sets to relatively compact sets.

\begin{proposition}\label{prop:atomiccounterexample}
Let $n\ge3$, $0<p\le Q$, and $0<q<\infty$. There is a finite positive
measure $\mu$ such that
$\mathcal A^p_{1,\alpha}\to L^q(\mu)$ is continuous and compact,
although
\[
 \sup_{a\in B}\frac{\mu(\Delta_r(a))}{\delta(a)^{Qq/p}}=\infty
\]
for every fixed $r\in(0,1)$.
\end{proposition}

\begin{proof}
For $p<Q$, choose $q<s<Qq/p$, put $t_j=2^{-j}$ and
$x_j=(1-t_j)e_1$, and set $\mu=\sum_{j\ge1}t_j^s\delta_{x_j}$, where $\delta_{x_j}$ denotes the Dirac measure at $x_j$.
Using the subcritical point-evaluation bound, we obtain
\[
 \int|f|^q\,d\mu\le C\|f\|_{p,\alpha}^q
       \sum_{j\ge1}t_j^{s-q}.
\]
The series converges, and its tails tend uniformly to zero on bounded
families. But the geometric ratio at $a=x_j$ is at least
$t_j^{s-Qq/p}\to\infty$.

For $p=Q$, choose $\beta>2Q/q$ and put
\[
 t_j=\exp(-j^\beta),\qquad x_j=(1-t_j)e_1,\qquad
 \mu=\sum_{j\ge1}\frac{t_j^q[\log(e/t_j)]^{q/Q}}{j^2}
                         \,\delta_{x_j}.
\]
This is a finite positive measure. By the critical evaluation bound,
\[
 \int|f|^q\,d\mu\le C\|f\|_{Q,\alpha}^q\sum_{j\ge1}j^{-2},
\]
again with uniformly vanishing tails on bounded families. At $a=x_j$,
the geometric ratio is bounded below by
$[\log(e/t_j)]^{q/Q}/j^2\to\infty$, by the choice of $\beta$.

In both cases, every bounded sequence has a locally uniformly
convergent subsequence whose limit belongs to the same Bergman
class. On finitely many atoms the convergence also holds in $L^q(\mu)$.
Together with the uniform tail estimate over the remaining atoms, this
proves compactness of the inclusion. If the original sequence converges
in the metric $d_p$, it converges locally uniformly, and the same
finite-atom and tail argument proves continuity.
\end{proof}

\section{Bounded and compact measure embeddings}\label{sec:embedding}

We record the forward inequalities for comparison with the reverse results. Nolder~\cite{Nolder90} proved unweighted embeddings. The weighted sufficient condition below follows from Oleinik--Vuorinen~\cite[Theorem~1.2 and Remark~2.13(4)]{OV}: take their submean function to be $u=|f|^p$, their embedding exponent to be $q/p\ge1$, and their weight exponent to be $\alpha$. Their conclusion is
\[
 \left(\int_Bu^{q/p}\,d\mu\right)^{p/q}
 \le C\int_Bu\,\delta^\alpha\,dx,
\]
which includes negative $\alpha$ and $p=q$.

Oleinik--Vuorinen also prove compact sufficiency for $p<q$ in Theorem~5.5 and distortion-dependent converses in Theorems~1.4 and~5.8, under their stated hypotheses. We include a short proof on the ball. It also covers compactness when $p=q$ and shows how concentration enters the converse arguments.

\begin{proposition}\label{thm:mainembedding}
Let $K\ge1$, $0<p\le q<\infty$, $\alpha>-1$, and let $\mu$ be a positive locally finite Borel measure on $B$. For a fixed $r\in(0,1)$, set
\[
 \Gamma_r(a)=\frac{\mu(\Delta_r(a))}{\delta(a)^{Qq/p}}.
\]
The condition
\begin{equation}\label{eq:carleson}
 \sup_{a\in B}\Gamma_r(a)<\infty
\end{equation}
implies
\begin{equation}\label{eq:embedding}
 \int_B|f|^q\,d\mu\le C\|f\|_{p,\alpha}^q
 \qquad(f\in\A^p_{K,\alpha}),
\end{equation}
and the inclusion is continuous in the ambient metrics. If
\begin{equation}\label{eq:vanishing}
 \Gamma_r(a)\longrightarrow0\qquad(|a|\to1),
\end{equation}
the inclusion is compact. Both converses hold if $K\ge\Kstar$ or if \eqref{eq:radialrange} holds. The two geometric conditions are independent of the fixed radius $r$.
\end{proposition}

\begin{proof}
First observe that local finiteness bounds $\Gamma_r$ on compact subsets: if $\delta(a)\ge t>0$, the balls $\Delta_r(a)$ lie in a common compact subset of $B$, and the denominator is bounded below. Thus \eqref{eq:vanishing} implies \eqref{eq:carleson}. To compare two fixed radii $r_1<r_2$, cover $\overline{B(0,r_2)}$ by finitely many hyperbolic balls of radius $r_1$ and apply $\phi_a$. Their new centers have boundary distances comparable to $\delta(a)$. This proves radius independence of boundedness and vanishing; the opposite implications follow by inclusion.

Choose a covering $\{\Delta_r(a_j)\}$ for which the enlargements $\Delta_R(a_j)$ have bounded overlap, where $r<R<1$, and put
\[
 A_j(f)=\int_{\Delta_R(a_j)}|f|^p\,dV_\alpha.
\]
Such a covering exists by Lemma~\ref{lem:hyperbolic}. For any set $J$ of indices, it follows from \eqref{eq:hyperbolicsup} and $q/p\ge1$ that
\begin{align*}
 \int_{\bigcup_{j\in J}\Delta_r(a_j)}|f|^q\,d\mu
 &\le C\sum_{j\in J}\Gamma_r(a_j)A_j(f)^{q/p}\\
 &\le C\left(\sup_{j\in J}\Gamma_r(a_j)\right)
       \left(\int_{\bigcup_{j\in J}\Delta_R(a_j)}
                          |f|^p\,dV_\alpha\right)^{q/p}.
\end{align*}
Bounded overlap is used in the second line; infinite sums follow by monotone convergence. Taking all indices, we obtain \eqref{eq:embedding}. By testing a nonzero constant, we also find that $\mu(B)<\infty$. If $\Delta_r(a_j)$ meets $\{\delta<s\}$, then $\delta(a_j)<C_0s$ and $\Delta_R(a_j)\subset\{\delta<C_0s\}$ for a fixed $C_0$. We may therefore apply the same estimate to obtain
\begin{equation}\label{eq:localizedembedding}
 \int_{\{\delta<s\}}|f|^q\,d\mu
 \le C\left(\sup_{\delta(a)<C_0s}\Gamma_r(a)\right)
       \left(\int_{\{\delta<C_0s\}}|f|^p\,dV_\alpha\right)^{q/p}.
\end{equation}

If $f_j\to f$ in the metric $d_p$, then $f_j\to f$ locally uniformly by Proposition~\ref{thm:complete}. Their $L^p_\alpha$ boundary tails are uniformly small: apply \eqref{eq:elementarypower} to $f_j=(f_j-f)+f$ for large $j$, and then handle the finitely many other terms. By formula~\eqref{eq:localizedembedding}, their $L^q(\mu)$ tails are uniformly small as well. On each compact subset, it follows from uniform convergence and local finiteness of $\mu$ that the sequence converges in $L^q(\mu)$. Applying \eqref{eq:elementarypower} to the difference on the remaining collar completes the proof of continuity.

Under \eqref{eq:vanishing}, we infer from formula~\eqref{eq:localizedembedding} that the $L^q(\mu)$ tails are uniformly small on every bounded Bergman family. Any bounded sequence has a locally uniformly convergent subsequence, whose limit lies in $\A^p_{K,\alpha}$ by Fatou's lemma. Together with the uniform tail estimate, compact convergence implies convergence in $L^q(\mu)$, as required.

For the converses, take the normalized mappings $h_a$ from Section~\ref{sec:concentration}; these mappings are also available by Proposition~\ref{prop:radialrange} under \eqref{eq:radialrange}. Their norms equal one. By \eqref{eq:peakbounds}, we have
\[
 \int_B|h_a|^q\,d\mu
 \ge c_r\delta(a)^{-Qq/p}\mu(\Delta_r(a)).
\]
Thus \eqref{eq:embedding} implies \eqref{eq:carleson}. A compact inclusion has a bounded image of the unit ball, so \eqref{eq:embedding} also follows by positive scaling. If \eqref{eq:vanishing} failed, there would be $a_j\to\partial B$ with $\Gamma_r(a_j)\ge\varepsilon>0$. A subsequence of $h_{a_j}$ would converge in $L^q(\mu)$, and a further subsequence would converge $\mu$-almost everywhere. But \eqref{eq:peaklocalzero} gives pointwise convergence to zero throughout $B$. The $L^q(\mu)$ limit would therefore be zero, contrary to the displayed lower bound. This proves necessity of \eqref{eq:vanishing}.
\end{proof}

\begin{corollary}\label{cor:weightedinclusions}
Let $0<p\le q<\infty$ and $\alpha,\gamma>-1$. The inclusion
\[
 \A^p_{K,\alpha}\longrightarrow\A^q_{K,\gamma}
\]
is bounded if
\[
 \frac{n+\gamma}{q}\ge\frac{n+\alpha}{p},
\]
and compact if the inequality is strict. These conditions are also necessary if $K\ge\Kstar$ or if \eqref{eq:radialrange} holds.
\end{corollary}

\begin{proof}
For $d\mu=dV_\gamma$, it follows from Lemma~\ref{lem:hyperbolic} that
\[
 \Gamma_r(a)\asymp\delta(a)^{n+\gamma-(n+\alpha)q/p}.
\]
The conclusion now follows from Proposition~\ref{thm:mainembedding}. Each mapping retains its distortion bound, so the target is the stated quasiregular class.
\end{proof}

\section{Compactness and approximation}\label{sec:further}

We finish with a few consequences of local compactness and the boundary estimates.

\subsection{Boundary tails and extremal problems}

\begin{proposition}\label{thm:compactcriterion}
A bounded family $\mathcal F\subset\A^p_{K,\alpha}$ is relatively compact in $d_p$ if and only if
\begin{equation}\label{eq:uniformtails}
 \lim_{t\downarrow0}\sup_{f\in\mathcal F}
       \int_{\{\delta<t\}}|f|^p\,dV_\alpha=0.
\end{equation}
\end{proposition}

\begin{proof}
Under \eqref{eq:uniformtails}, every sequence in $\mathcal F$ has a subsequence $(f_j)$ converging locally uniformly to a $K$-quasiregular mapping $f$. It follows from Fatou's lemma that $f$ belongs to the Bergman class, and the same tail bound holds for $f$. On $\{\delta\ge t\}$ the integral of $|f_j-f|^p$ tends to zero by uniform convergence. On $\{\delta<t\}$ it is bounded, by \eqref{eq:elementarypower}, by $c_p$ times the two tails. First choose $t$ small and then let $j\to\infty$. The subsequence converges in $d_p$.

Conversely, by relative compactness, for each $\varepsilon>0$ we can choose finitely many mappings $g_1,\ldots,g_N$ such that every $f\in\mathcal F$ satisfies $\|f-g_i\|_{p,\alpha}^p<\varepsilon$ for some $i$. Choose $t$ so small that all the tails of the $g_i$ are smaller than $\varepsilon$. Then
\[
 \int_{\{\delta<t\}}|f|^p\,dV_\alpha
 \le c_p\left(\|f-g_i\|_{p,\alpha}^p+
              \int_{\{\delta<t\}}|g_i|^p\,dV_\alpha\right)
 <2c_p\varepsilon,
\]
which proves \eqref{eq:uniformtails}.
\end{proof}

\begin{corollary}\label{cor:lowerexponent}
If $0<p<q<\infty$, the inclusion $\A^q_{K,\alpha}\to\A^p_{K,\alpha}$ is continuous and maps bounded sets to relatively compact sets.
\end{corollary}

\begin{proof}
Continuity in the ambient metrics follows from H\"older's inequality, which also implies
\[
 \int_{\{\delta<t\}}|f|^p\,dV_\alpha
 \le\|f\|_{q,\alpha}^p
       V_\alpha(\{\delta<t\})^{1-p/q}.
\]
The last factor tends to zero because $\alpha>-1$. Apply Proposition~\ref{thm:compactcriterion}.
\end{proof}

\begin{proposition}\label{prop:extremals}
For every $a\in B$, the point-evaluation supremum $\mathcal E_{K,p,\alpha}(a)$ is finite and attained. More generally, if finitely many prescribed point values are attained by a member of $\A^p_{K,\alpha}$, they are attained by a mapping minimizing $\|f\|_{p,\alpha}$ among all admissible mappings.
\end{proposition}

\begin{proof}
By the estimate \eqref{eq:evaluationgrowth}, the supremum is finite. An extremizing sequence in the unit ball has a locally uniformly convergent subsequence. Its limit $f$ belongs to the same quasiregular class and satisfies $\|f\|_{p,\alpha}\le1$ by Fatou's lemma, while evaluation at $a$ passes to the limit. For interpolation, take an admissible sequence $(f_j)$ such that $\|f_j\|_{p,\alpha}$ tends to the infimum. By the same argument, the limit preserves all the prescribed values and attains that infimum.
\end{proof}

\subsection{Approximation by dilations}

Dilations provide approximations that retain the distortion bound.

\begin{proposition}\label{thm:dilation}
If $f\in\A^p_{K,\alpha}$ and $f_r(x)=f(rx)$, then
\[
 \|f_r-f\|_{p,\alpha}\longrightarrow0\qquad(r\uparrow1).
\]
Hence bounded $K$-quasiregular mappings extending to a neighborhood of $\overline B$ are dense in $\A^p_{K,\alpha}$.
\end{proposition}

\begin{proof}
Put $\tau=1-r<1/16$. If $\delta(x)\ge2\tau$, then $|rx-x|\le\delta(x)/2$, so $|f_r(x)|\le\mathcal M_{1/2}f(x)$. By the maximal estimate and \eqref{eq:elementarypower}, there is an integrable majorant for
\[
 \boldsymbol 1_{\{\delta\ge2\tau\}}|f_r-f|^p.
\]
The integral over this region therefore tends to zero by dominated convergence.

On the remaining collar, $\tau\le\delta(rx)<3\tau$. Using the local estimate on $B(rx,\tau/4)$, where the weight is comparable to $\tau^\alpha$, we obtain
\[
 |f(rx)|^p\le C\tau^{-Q}
        \int_{B(rx,\tau/4)}|f(y)|^p\,dV_\alpha(y).
\]
For fixed $y$, the set of $x$ with $\delta(x)<2\tau$ and $|rx-y|<\tau/4$ has diameter at most $\tau/(2r)<\tau$. When nonempty, it lies in a cap of radius $C\tau$ centered at a nearest boundary point to one of its members. By Lemma~\ref{lem:cap}, its weighted volume is at most $C\tau^Q$. Moreover, each such $y$ satisfies $\delta(y)<4\tau$. Applying Tonelli's theorem, we now obtain
\[
 \int_{\{\delta<2\tau\}}|f_r|^p\,dV_\alpha
 \le C\int_{\{\delta<4\tau\}}|f|^p\,dV_\alpha\longrightarrow0.
\]
The corresponding tail of $f$ also tends to zero. By \eqref{eq:elementarypower}, convergence therefore holds on the collar as well. Finally, $f_r$ has the same distortion bound, is defined on $B(0,1/r)$, and is bounded on $\overline B$.
\end{proof}

\subsection{Further questions}

Corollary~\ref{cor:criticaldistortion} separates the conformal case from
every fixed $K>1$ at $p=Q$. In dimensions $n\ge3$, the results above
leave the subcritical range
\[
 0<p<Q,\qquad 1<K<\Kstar,\qquad pK^{1/(n-1)}\le Q
\]
to be understood. It would be useful to determine the sharp
point-evaluation growth and the geometric condition for domination in
this range.

For forward embeddings with $0<q<p$, choose a separated hyperbolic
covering $\{\Delta_r(a_j)\}$ as in Lemma~\ref{lem:hyperbolic}. The condition
\[
 \left(\frac{\mu(\Delta_r(a_j))}{\delta(a_j)^{Qq/p}}\right)_j
 \in\ell^{p/(p-q)}
\]
is sufficient for \eqref{eq:embedding} for every fixed $K$, by
Oleinik--Vuorinen~\cite[Theorem~4.1 and Remark~2.13(4)]{OV}. In the disk
it is also necessary: an inequality on the full quasiregular class holds
in particular for analytic functions, so the classical weighted Bergman
theorem applies; see Pau--Zhao~\cite[Theorem~B]{PZ}. Their weight
$(1-|z|^2)^\alpha$ is comparable to $\delta(z)^\alpha$.

In dimensions $n\ge3$, a necessary condition must also reflect the
distortion class. Proposition~\ref{prop:atomiccounterexample} applies when
$q<p$: for $K=1$ and $p\le Q$, even the uniform local mass bound can
fail to be necessary for a compact embedding. The question is therefore
whether the summability condition is necessary in classes admitting
stronger concentration. Single pole mappings test one region at a time;
the analytic necessity argument uses sums, which need not preserve a
fixed quasiregular distortion bound.

For reverse inequalities, Theorem~\ref{thm:singularreverse} treats measures
whose small-ball masses satisfy a uniform decay estimate, while
Proposition~\ref{thm:sampling} gives atomic norming measures for every
fixed-distortion class. A natural further question is to find a geometric
characterization that includes both kinds of measures.

\appendix

\section{A Zorich map and one inverse branch}\label{app:zorich}

For completeness, we give the elementary construction used in Proposition~\ref{thm:poles}. It is the standard reflected-beam construction of Zorich~\cite{Zorich}; see also~\cite[Section~I.3.3]{Rickman} and~\cite[Section~3]{Nicks}. The estimates below give a distortion bound depending only on the dimension.

\begin{lemma}\label{lem:zorich}
For $n\ge3$, there are a locally Lipschitz quasiregular mapping
\[
 Z:\R^{n-1}\times\R\longrightarrow\R^n\setminus\{0\}
\]
and a quasiconformal homeomorphism
\[
 \Lambda:H_+\longrightarrow(-1,1)^{n-1}\times\R,
 \qquad H_+=\{v\in\R^n:v_n>0\},
\]
such that
\begin{equation}\label{eq:zorichmodulus}
 |Z(u,s)|=e^s,\qquad Z\circ\Lambda=\mathrm{id}_{H_+},
 \qquad \Lambda_n(v)=\log|v|.
\end{equation}
Their outer distortions are at most $L_n^{2(n-1)}$, where
\[
 L_n=\max\{3\pi/2,\,n-1+2\sqrt{n-1}\}.
\]
\end{lemma}

\begin{proof}
Put $d=n-1$ and $C=[-1,1]^d$. We construct a bi-Lipschitz homeomorphism $h$ from $C$ onto the closed upper hemisphere of the unit sphere. The radial map
\[
 P(u)=\frac{\|u\|_\infty}{|u|}u\quad(u\ne0),\qquad P(0)=0,
\]
maps $C$ onto the closed unit ball of $\R^d$. Its inverse is
\[
 P^{-1}(v)=\frac{|v|}{\|v\|_\infty}v\quad(v\ne0).
\]
Both are Lipschitz. For example, using the triangle inequality and
$\bigl|\|u\|_\infty-\|v\|_\infty\bigr|\le|u-v|$, we obtain a Lipschitz constant $3$ for $P$. The same calculation applied to the inverse, with $\|u\|_\infty\ge|u|/\sqrt d$, leads to the bound $d+2\sqrt d$ for $P^{-1}$.

Map the unit ball of $\R^d$ onto the upper hemisphere by
\[
 H(v)=\left(\frac{\sin(\pi|v|/2)}{|v|}v,\ \cos(\pi|v|/2)\right),
\]
with its continuous value at zero. To check the bi-Lipschitz property, write $v=r\theta$ and $w=s\zeta$. Using the identity
\begin{align*}
 |H(v)-H(w)|^2
 &=4\sin^2\frac{\pi(r-s)}4
   +\sin\frac{\pi r}2\sin\frac{\pi s}2|\theta-\zeta|^2
\end{align*}
together with the inequalities $r\le\sin(\pi r/2)\le\pi r/2$ for $0\le r\le1$, we obtain
\[
 |v-w|\le|H(v)-H(w)|\le\frac\pi2|v-w|.
\]
The cases $r=0$ or $s=0$ follow by continuity. Thus $h=H\circ P$ is bi-Lipschitz. Its boundary is mapped to the equator. We may therefore take
\[
 L_n=\max\{3\pi/2,\,n-1+2\sqrt{n-1}\}
\]
as a common bi-Lipschitz constant for $h$ and $h^{-1}$. We choose the orientation of $h$ so that the map below is orientation preserving; precomposition with one coordinate reflection, if needed, does not change $L_n$.

On the beam $C\times\R$, define
\[
 Z(u,s)=e^sh(u).
\]
On its interior this is a homeomorphism onto $H_+$. At almost every $u$, the vectors in the range of $Dh(u)$ are orthogonal to $h(u)$, because $|h(u)|=1$. Consequently,
\[
 |DZ(u,s)(v,\tau)|^2
 =e^{2s}\bigl(|Dh(u)v|^2+\tau^2\bigr).
\]
Every singular value of $DZ$ lies between $L_n^{-1}e^s$ and $L_ne^s$. In particular, its outer distortion and the outer distortion of its inverse are at most $L_n^{2(n-1)}$.

We extend $Z$ to all beams by reflection. For $m=(m_1,\ldots,m_d)\in\mathbb Z^d$, let
\[
 C_m=2m+C,\qquad
 q_m(u)_j=(-1)^{m_j}(u_j-2m_j),\qquad
 S=\operatorname{diag}(1,\ldots,1,-1).
\]
For $u\in C_m$, set
\begin{equation}\label{eq:reflectionextension}
 Z(u,s)=e^sS^{m_1+\cdots+m_d}h(q_m(u)).
\end{equation}
The definitions agree on common faces. Indeed, the folded coordinates agree there, their images under $h$ lie on the equator, and $S$ fixes the equator pointwise. The formula therefore defines a continuous mapping. On each compact subset only finitely many beams occur. Their Lipschitz restrictions can therefore be glued, so the extension is locally Lipschitz across the faces as well.

Each domain reflection changes orientation once, and the accompanying target reflection changes it once more. Thus the Jacobian in \eqref{eq:reflectionextension} is positive almost everywhere. Orthogonal reflections preserve singular values, so the same distortion bound holds on every beam. The union of the beam faces has volume zero. Hence the extension is quasiregular and still satisfies $|Z(u,s)|=e^s$.

Finally, on $H_+$ the inverse of the original branch is
\[
 \Lambda(v)=\left(h^{-1}\left(\frac v{|v|}\right),\log|v|\right).
\]
It is locally bi-Lipschitz and quasiconformal, with outer distortion at most $L_n^{2(n-1)}$. Formula \eqref{eq:zorichmodulus} follows immediately.
\end{proof}

\section*{AI Declaration}
ChatGPT (OpenAI) was used for language editing, organisation, and assistance with finite symbolic calculations. The author independently checked all mathematical contents and takes full responsibility for the final manuscript.

\end{document}